%% file: manuscripts.tex
\documentclass[final,onefignum,onetabnum]{siamonline171218}

\input{ex_shared}

\ifpdf
\hypersetup{
  pdftitle={\TheTitle},
  pdfauthor={\TheAuthors}
}
\fi

\begin{document}

\maketitle

% REQUIRED
\begin{abstract}
We generalize the successive continuation paradigm introduced by Kern{\'e}vez and Doedel~\cite{DKK91} for locating locally optimal solutions of constrained optimization problems to the case of simultaneous equality and inequality constraints. The analysis shows that potential optima may be found at the end of a sequence of easily-initialized separate stages of continuation, without the need to seed the first stage of continuation with nonzero values for the corresponding Lagrange multipliers. A key enabler of the proposed generalization is the use of complementarity functions to define relaxed complementary conditions, followed by the use of continuation to arrive at the limit required by the Karush-Kuhn-Tucker theory. As a result, a successful search for optima is found to be possible also from an infeasible initial solution guess. The discussion shows that the proposed paradigm is compatible with the staged construction approach of the \textsc{coco} software package. This is evidenced by a modified form of the \textsc{coco} core used to produce the numerical results reported here. These illustrate the efficacy of the continuation approach in locating stationary solutions of an objective function along families of two-point boundary value problems and in optimal control problems.
\end{abstract}

% REQUIRED
\begin{keywords}
constrained optimization, feasible solutions, complementarity conditions, boundary-value problems, periodic orbits, optimal control, successive continuation, software implementation
\end{keywords}

% REQUIRED
\begin{AMS}
  49K15, 49K27, 49M05, 49M29, 90C33, 45J05, 34B99
\end{AMS}

\section{Introduction}
\label{sec:intro}

% start from general backgrounds, two methods, direct and indirect methods to solving optimization problems, advantages and disadvantages of two approaches, goes to necessarily conditions; initial conditions for adjoints, continuation schemes, limitations (without inequality constraints well tackled), objetive of this paper, basic ideas, and supportive of staged construction

% https://link.springer.com/content/pdf/10.1007%2FBF02071065.pdf
%Direct and indirect methods for trajectory optimization
%Survey of Numerical Methods for Trajectory Optimization
% high-dimensional, adjoint - senstivity analysis

Parameter continuation techniques enable a global study of smooth manifolds of solutions to underdetermined systems of equations. It stands to reason that they should also be useful for optimization problems constrained to such manifolds. Starting from a single solution and local information about the governing equations, such techniques generate a successively expanding, suitably dense cover of solutions, among which optima may be sought. Suggestive examples include optimization along families of periodic or quasiperiodic solutions of nonlinear dynamical systems or optimal control problems with end-point constraints.

As shown first by Kern{\'e}vez and Doedel~\cite{DKK91,Doedel-ii}, parameter continuation techniques may be effectively deployed as a core element of a search strategy for optima along a constraint manifold. This is accomplished by seeking simultaneous solutions to the original set of equations and a set of additional adjoint conditions, linear and homogeneous in a set of unknown Lagrange multipliers. The technique introduced in~\cite{DKK91} demonstrates how local optima may be found at the terminal points of sequences of continuation runs, with each successive run initialized by the final solution found in the previous run. Remarkably, by the linearity and homogeneity of the adjoint conditions, the initial runs may be conveniently initialized with vanishing Lagrange multipliers from solutions to the original set of equations. Kern{\'e}vez and Doedel's technique thereby overcomes the difficulty of determining values for the original unknowns \textit{and the Lagrange multipliers} that provide an adequate initial solution guess to the necessary conditions for local optima.

The objective of this paper is to extend Kern{\'e}vez and Doedel's technique to optimization problems with simultaneous equality and inequality constraints. This {nontrivial generalization} is here accomplished by seeking simultaneous solutions to the original set of equations, additional adjoint conditions that are again linear and homogeneous in a set of unknown Lagrange multipliers, and relaxed nonlinear complementarity conditions in terms of the inequality functions and the corresponding Lagrange multipliers. As before, local optima are found at the terminal points of sequences of consecutive continuation runs with linearity and homogeneity again enabling initialization with vanishing Lagrange multipliers. Notably, here, certain stages of continuation are used to drive the relaxation parameters to zero in order to ensure that the necessary nonlinear complementarity conditions are exactly satisfied.

In the absence of inequality constraints, Kern{\'e}vez and Doedel's technique relies on the existence of a branch point for the initial continuation problem from which emanate two one-dimensional branches of solutions with vanishing and non-vanishing Lagrange multipliers, respectively. As shown here, such a branch point may also be found in the presence of inequality constraints. Interestingly, inequality constraints afford additional opportunities for generating solutions with non-vanishing Lagrange multipliers from an initial solution with all zero multipliers. Remarkably, in the presence of inequality constraints, the successive continuation technique may even benefit from initialization on solutions that violate these constraints. {An original contribution of this manuscript is the formulation and proof of several key lemmas that establish these properties for large classes of optimization problems.}

Example applications of Kern{\'e}vez and Doedel's technique can be found in \cite{bio,bio1,acoustic1,acoustic}. In each case, the governing set of equations, including the adjoint conditions, forms a two-point boundary-value problem that is analyzed using the software package \textsc{auto}~\cite{auto}. {Such an implementation is also possible for the extension to the presence of inequality constraints, as nothing in the formulation relies on a particular software implementation. Nevertheless}, recent work by the present authors~\cite{staged_adjoint} demonstrates the implementation of a staged construction approach for the adjoint equations in the \textsc{matlab}-based software platform \textsc{coco}~\cite{coco}, supporting the assembly of the full continuation problem from partial problems with predefined structure and adjoints. A powerful example of this functionality is the optimal design of a transfer trajectory between two halo orbits near a libration point of the circular restricted three-body problem~\cite{staged_adjoint}. In this paper, we use a further extension of \textsc{coco} to locate stationary points of an algebraic-integral objective functional constrained by a two-point boundary-value problem and an integral inequality, as well as for an optimal control problem under integral inequality constraints.

We finally note a further use of continuation methods when seeking solutions to singular, non-smooth, or non-differentiable optimization problems as limits of continuous families of regularized optimization problems. Such applications will not be considered here, but include regularized optimal control problems with differentiable solutions that approach discontinuous bang-bang control solutions in the singular limit~\cite{bang-bang-1,bang-bang-2}, as well as regularized optimal control problems with index-1 differential-algebraic constraints that converge to higher-index constraints in the singular limit~\cite{fabien2014indirect}. Such regularizations can, in principle, be combined with the techniques described in this paper, provided that the singular limits can be reached in the corresponding stages of continuation.

% organization of rest paper
The remainder of this paper is organized as follows. In \cref{sec:prob-stat}, we formulate the first-order necessary conditions for local optima of an objective function in the presence of equality and real-valued inequality constraints on a general Banach space. We describe the use of a nonsmooth complementarity function to enforce the complementary slackness conditions on the inequality functions and the corresponding Lagrange multipliers. \Cref{sec:example} presents the application of the successive continuation technique {to} a finite-dimensional optimization problem {that motivates the subsequent theoretical development. The detailed analysis illustrates the additional flexibility afforded by the presence of inequality constraints and lays the foundation for an algorithmic implementation in numerical software.} The generalization to the infinite-dimensional context is discussed in \cref{sec:cont}, with reference to several {original} key lemmas that are proved in \cref{sec:essential lemmas}. After presenting some implementation details in \cref{sec:implementation} {that are particular to the advantages afforded by the staged construction paradigm of \textsc{coco}}, we consider additional examples in \cref{sec:application} to demonstrate the effectiveness of the proposed {optimization} approach. A brief summary and several directions for future research are considered in the concluding \cref{sec:conclusions}.

\section{Problem statement}
\label{sec:prob-stat}
Consider the problem of finding a locally unique pair $(\hat{u},\hat{\mu})$ that is a stationary point of the function $(u,\mu)\mapsto\mu_1$ under the equality constraints $F(u,\mu)=0$ and inequality constraints $G(u)\leq0$, where
\begin{equation}
F(u,\mu)\mapsto\begin{pmatrix}\Phi(u)\\\Psi(u)-\mu\end{pmatrix}.
\end{equation}
Here, $\Phi:U\rightarrow Y$, $\Psi:U\rightarrow\mathbb{R}^l$ and $G:U\rightarrow\mathbb{R}^q$ are continuously Fr\'{e}chet differentiable mappings, and $U$ and $Y$ are real Banach spaces with duals $U^*$ and $Y^*$. We refer to the set $\{u\in U:G(u)\leq 0\}$ as the feasible region and to its complement as the infeasible region.

Let $\mathbb{A}:=\{i:G_i(\hat{u})=0\}$ denote the set of indices of active inequality constraints evaluated at $\hat{u}$, and suppose that the range of $(D\Phi(\hat{u}),DG_{\mathbb{A}}(\hat{u}))$ equals $ Y\times \mathbb{R}^{\left |\mathbb{A}\right |}$. It follows from Corollary 9.4 in~\cite{ben1982unified} that there exist unique $\hat{\lambda}\in Y^*$, $\hat{\eta}\in\mathbb{R}^l$, and $\hat{\sigma}\in\mathbb{R}^q$ that satisfy the generalized Karush-Kuhn-Tucker (KKT) optimality conditions
\begin{equation}
\Phi(\hat{u})=0,\quad\Psi(\hat{u})-\hat{\mu}=0,
\end{equation}
\begin{equation}
(D\Phi(\hat{u}))^\ast\hat{\lambda}+(D\Psi(\hat{u}))^\ast\hat{\eta}+(DG(\hat{u}))^\ast\hat{\sigma}=0,\quad\hat{\eta}_1=1,\quad\hat{\eta}_{\{2,\ldots,l\}}=0,
\end{equation}
and
\begin{equation}
\label{eq:ncp}
\hat{\sigma}_i\geq0, \quad -G_i(\hat{u})\geq0,\quad \hat{\sigma}_iG_i(\hat{u})=0,\quad 1\leq i\leq q,
\end{equation}
where $(D\Phi(u))^\ast:Y^*\rightarrow U^*$, $(D\Psi(u))^\ast:\mathbb{R}^l\rightarrow U^*$ and $(DG(u))^\ast:\mathbb{R}^q\rightarrow U^*$ are the adjoints of the Fr\'{e}chet derivatives $D\Phi(u)$, $D\Psi(u)$  and $DG(u)$, respectively.

Inspired by the general theory of nonlinear complementarity problems~\cite{comp,cp_review,ncp}, we find it convenient to convert \eqref{eq:ncp} into a set of nonlinear equations. Specifically, let $\chi:\mathbb{R}\times\mathbb{R}\to\mathbb{R}$ be a function that satisfies the conditions 
\begin{equation}
    \chi(a,b)=0 \iff a,b\geq0,\quad ab=0.
\end{equation}
Then \eqref{eq:ncp} is equivalent to the condition
\begin{equation}
\label{eq: equiv_ncp}
  K(\hat{\sigma},-G(\hat{u})):=\begin{pmatrix}\chi(\hat{\sigma}_1,-G_1(\hat{u})) & \ldots & \chi(\hat{\sigma}_q,-G_q(\hat{u}))\end{pmatrix}^\top =0.
\end{equation}

In this paper, we let $\chi$ equal the Fischer-Burmeister function
\begin{equation}
    \chi(a,b) = \sqrt{a^2+b^2}-a-b,
\end{equation}
whose contour plot is shown in Fig.~\ref{fig:ncp}. In particular, for $\kappa>0$,
\begin{equation}\label{chiform}
\chi(a,b)=\kappa\quad\Rightarrow\quad a=-\frac{\kappa(2b+\kappa)}{2(b+\kappa)},\quad b>-\kappa.
\end{equation}
As long as $(a,b)\ne (0,0)$,
\begin{equation}
\label{eq:phi_ab}
    \chi_a(a,b) :=\frac{\partial \chi}{\partial a}(a,b) = \frac{a}{\sqrt{a^2+b^2}}-1,\quad \chi_b(a,b) :=\frac{\partial \chi}{\partial b}(a,b) = \frac{b}{\sqrt{a^2+b^2}}-1,
\end{equation}
from which we conclude that
\begin{equation}
\label{eq:derivative-b}
    \chi_a(0,b)=-1,\quad \chi_b(0,b) = \sgn(b)-1,
\end{equation}
for $b\ne 0$. In particular, $\chi_b(0,b)=0$ when $b>0$. The function $\chi$ is clearly singular at $(0,0)$.
\begin{figure}[ht!]
\centering
\includegraphics[width=0.65\textwidth]{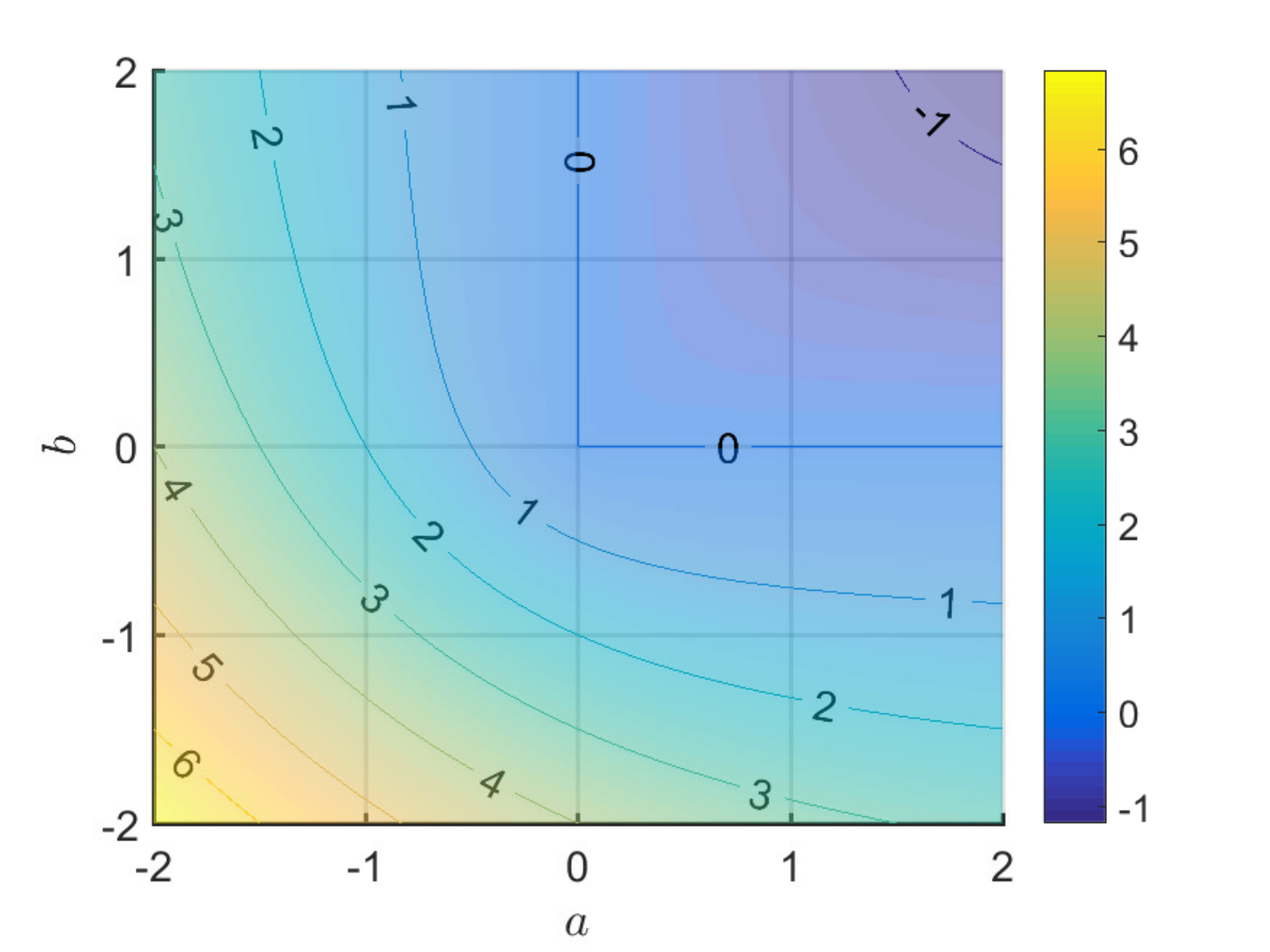}
\caption{Contour plot of the Fischer-Burmeister function $\chi(a,b)=\sqrt{a^2+b^2}-a-b$.}
\label{fig:ncp}
\end{figure}

\section{Motivating example}\label{sec:example}

Consider the problem of locating a local minimum of the function $u:=(x,y)\mapsto \Psi_1(u):=(x-2)^2+2(y-1)^2$ subject to the inequalities
\begin{equation}
G_1(u):=x+4y-3\leq0,\quad G_2(u):=-x+y\leq0.
\end{equation}
In the notation of the previous section, $U=\mathbb{R}^2$, $\mathbb{Y}=\emptyset$, $l=1$, and $q=2$. By the Karush-Kuhn-Tucker conditions, there exist unique non-negative scalars $\hat{\sigma}_1$ and $\hat{\sigma}_2$ such that
\begin{gather}
2\begin{pmatrix}\hat{x}-2\\2\hat{y}-2\end{pmatrix}+\begin{pmatrix}1\\4\end{pmatrix}\hat{\sigma}_1+\begin{pmatrix}-1\\1\end{pmatrix}\hat{\sigma}_2=0,\quad
\hat{\sigma}_1(\hat{x}+4\hat{y}-3)=0,\quad
\hat{\sigma}_2( -\hat{x}+\hat{y})=0.\label{KKT1}
\end{gather}
It follows that a candidate stationary point in the feasible region is located at $\hat{u}=\left(5/3,1/3\right)$, since
\begin{itemize}
\item $\hat{\sigma}_1=\hat{\sigma}_2=0$ only if $\hat{u}=(2,1)$, but $G_1(2,1)=3\nleq0$;
\item  $\hat{\sigma}_1=0,\hat{\sigma}_2\neq0$ only if $\hat{u}=\left(4/3,4/3\right)$ and $\hat{\sigma}_2=-4/3\ngeq 0$;
\item $\hat{\sigma}_1\neq0,\hat{\sigma}_2=0$ only if $\hat{u}=\left(5/3,1/3\right)$  and $\hat{\sigma}_1=2/3\geq 0$;
\item $\hat{\sigma}_1\neq0,\hat{\sigma}_2\neq0$ only if $\hat{u}=\left(3/5,3/5\right)$, $\hat{\sigma}_1=22/25$, and $\hat{\sigma}_2=-48/25\ngeq 0$.
\end{itemize}
This point lies on the boundary of the feasible region defined by $G_1=0$ and $G_2<0$, along which $\Psi_1$ evaluates to $3(6y^2-4y+1)$ and attains a minimum at $\hat{y}=1/3$. Moreover, for $\epsilon\ll 1$, $\Psi_1(5/3+\epsilon v,1/3+\epsilon w)\approx 1-2G_1(5/3+\epsilon v,1/3+\epsilon w)/3$. We conclude that $\hat{u}$ is a unique local minimum of $\Psi_1$ in the feasible region.

We illustrate next a method for locating the minimum at $\hat{u}$ using a successive continuation approach that connects an initial point $u_0$ to $\hat{u}$ via a sequence of intersecting one-dimensional manifolds. To this end, consider the following four regions: $U_{+/-}=\{u\in\mathbb{R}^2:G_1(u)>0, G_2(u)<0\}$, $U_{-/+}=\{u\in\mathbb{R}^2:G_1(u)<0, G_2(u)>0\}$,  $U_{+/+}=\{u\in\mathbb{R}^2:G_1(u)>0, G_2(u)>0\}$, and $U_{-/-}=\{u\in\mathbb{R}^2:G_1(u)<0, G_2(u)<0\}$. It follows that $U_{+/-}$, $U_{-/+}$, and $U_{+/+}$ are open subsets of the infeasible region, while $U_{-/-}$ is an open subset of the feasible region.

Suppose that $u_0\in U_{+/-}$ and let $\kappa_{0,1}:=\chi\left(0,-G_1(u_0)\right)=2G_1(u_0)>0$. It follows that $u_0$ lies on a locally unique one-dimensional solution manifold of the equation
\begin{equation}\label{exred}
\chi\left(0,-G_1(u)\right)-\kappa_{0,1}=0\quad\Longleftrightarrow\quad G_1(u)=G_1(u_0).
\end{equation}
The point
\begin{equation}
u_1:=\left( \frac{15+G_1(u_0)}{9}, \frac{3+2G_1(u_0)}{9}\right)
\end{equation}
is a stationary point of $\Psi_1$ on this manifold and $G_2<0$ along the entire segment of the manifold between $u_0$ and $u_1$ provided that $G_1(u_0)<12$.

Consider next the system of equations
\begin{gather}\label{exrest}
\left\{\begin{array}{c}
\Psi_1(u)-\mu_1=0,\\
\left(D\Psi_1(u)\right)^\top\eta_1+\left(DG(u)\right)^\top\sigma=0,\\
\chi\left(\sigma_1,-G_1(u)\right)-\kappa_1=0,\\
\chi\left(\sigma_2,-G_2(u)\right)-\kappa_2=0,\end{array}\right.
\end{gather}
with $\kappa_1=\kappa_{0,1}$, $\kappa_2=0$, and unknowns $(u,\mu_1,\eta_1,\sigma_1,\sigma_2)$. Then, every solution $u\in U_{+/-}$ of \eqref{exred} corresponds to a solution $(u,\Psi_1(u),0,0,0)$ of \eqref{exrest}. Indeed, for every such point with $u\ne u_1$, the corresponding Jacobian is found to have full rank. Thus, provided that $u_0\ne u_1$, the one-dimensional solution manifold of \eqref{exred} through $u_0$ corresponds to a locally unique one-dimensional solution manifold of \eqref{exrest} through $(u_0,\Psi_1(u_0),0,0,0)$.

Since $u_1$ is a stationary point of $\Psi_1$ along a level curve of $G_1$, the matrices $D\Psi_1(u_1)$ and $DG_1(u_1)$ are linearly dependent. It follows that if $G_1(u_0)<12$ then $(u_1,\Psi_1(u_1),0,0,0)$ is a branch point of \eqref{exrest} through which passes a secondary one-dimensional solution manifold, locally parameterized by $\eta$, along which $\sigma_2=0$ and the matrices $D\Psi_1(u)$ and $DG_1(u)$ are linearly dependent. Substitution yields $y=2x-3$ and $\sigma_1=2(2-x)\eta$, where $x$ is implicitly defined by
\begin{equation}
\chi\left(2(2-x)\eta,15-9x\right)-\kappa_{0,1}=0.
\end{equation}
Since $G_1(u_0)<12$, it follow from \eqref{chiform} that $G_2<0$ along this manifold for $\eta\in[0,1]$. Denote the corresponding $u$ for $\eta=1$ by $u_2$. Then, $(u_2,\Psi_1(u_2), 2(2-x_2),0,\kappa_{0,1})$ is a solution to \eqref{exrest} with $\eta=1$, $\kappa_2=0$, and unknowns $(u,\mu_1,\sigma_1,\sigma_2,\kappa_1)$.  Indeed, this point lies on a locally unique one-dimensional manifold of such solutions, along which $\sigma_2=0$, $y=2x-3$, $\sigma_1=2(2-x)$, where $x$ is implicitly defined by
\begin{equation}
\chi\left(2(2-x),15-9x\right)-\kappa_1=0.
\end{equation}
Since $G_1(u_0)<12$, it again follows from \eqref{chiform} that $G_2<0$ along this manifold for $\kappa_1\in[0,\kappa_{0,1}]$. The corresponding $u$ for $\kappa_1=0$ then equals $\hat{u}$, as expected.

Numerical results using parameter continuation with the \textsc{coco} software package validate the above analysis.
%{Pseudo-arclenth algorithm is used to perform the numerical continuation}
With $u_0=(3,1)\in U_{+/-}$, we have $G_1(u_0)=4$ and $\kappa_{0,1}=8$. Continuation along the solution manifold to \eqref{exrest} with $\kappa_1=\kappa_{0,1}$ and $\kappa_2=0$ results in a curve along which a local minimum of $\Psi_1$ is detected at $(x,y,\mu_1,\eta_1,\sigma_1,\sigma_2)=(2.1111, 1.2222,0.1111,0,0,0)$ represented by the red dots in~\cref{fig:cont-path-L1}. Branch switching at the stationary point results in the secondary branch terminating at the point $(2.0997, 1.1995, 0.0895, 1.0000,  -0.1995, 0)$ represented by the blue dots in~\cref{fig:cont-path-L1}. Finally, continuation along the solution manifold to \eqref{exrest} with $\eta=1$ and $\kappa_2=0$ yields the curves in~\cref{fig:cont-path-L1} connecting the blue dots to the black dots at $(x,y,\mu_1,\sigma_1,\sigma_2,\kappa_1)=(1.6667,0.3333,1.0000,0.6667,0,0)$, consistent with the analytical solution.

\begin{figure}[h]
\centering
\includegraphics[width=3.0in]{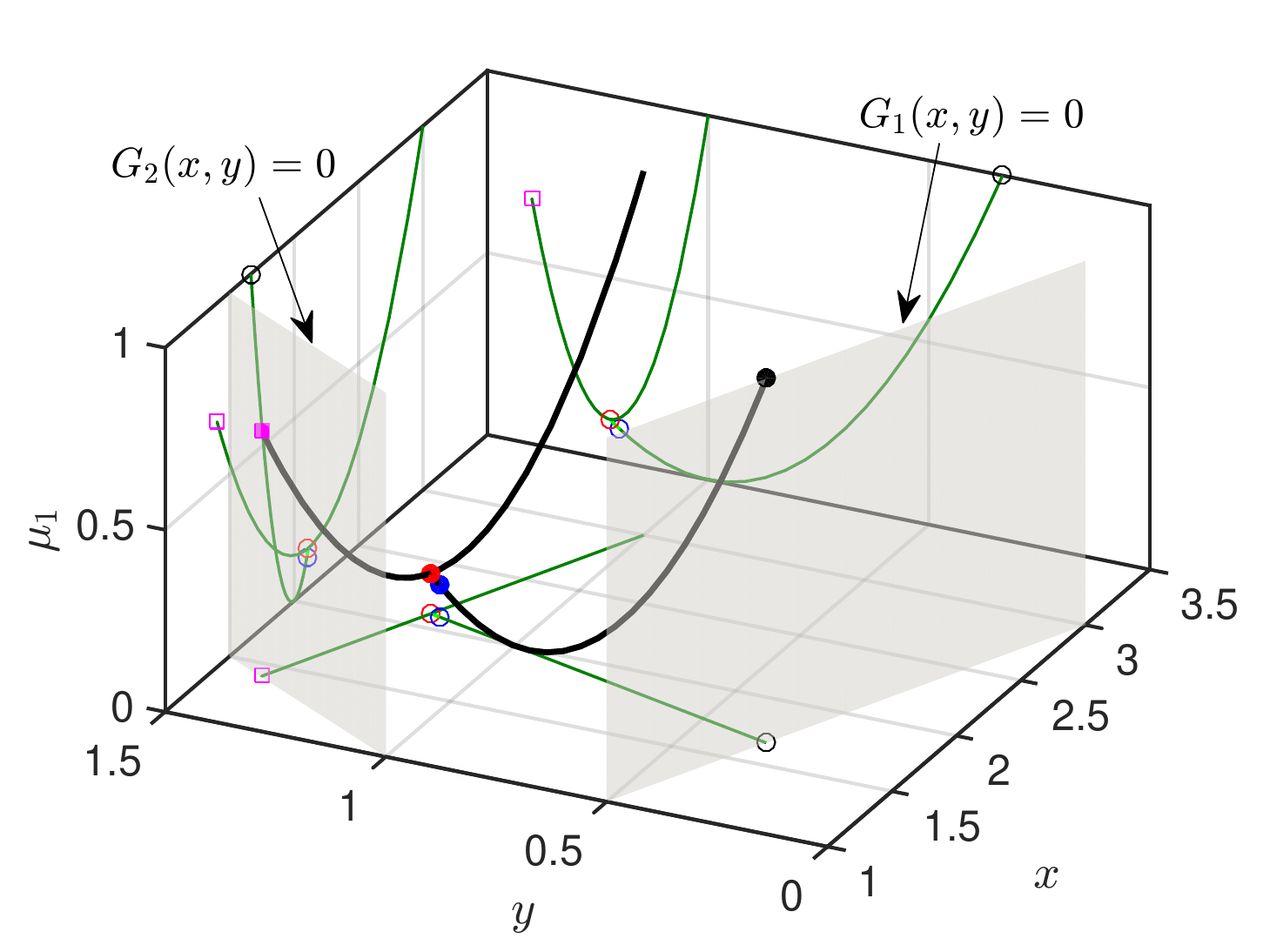}
\includegraphics[width=3.0in]{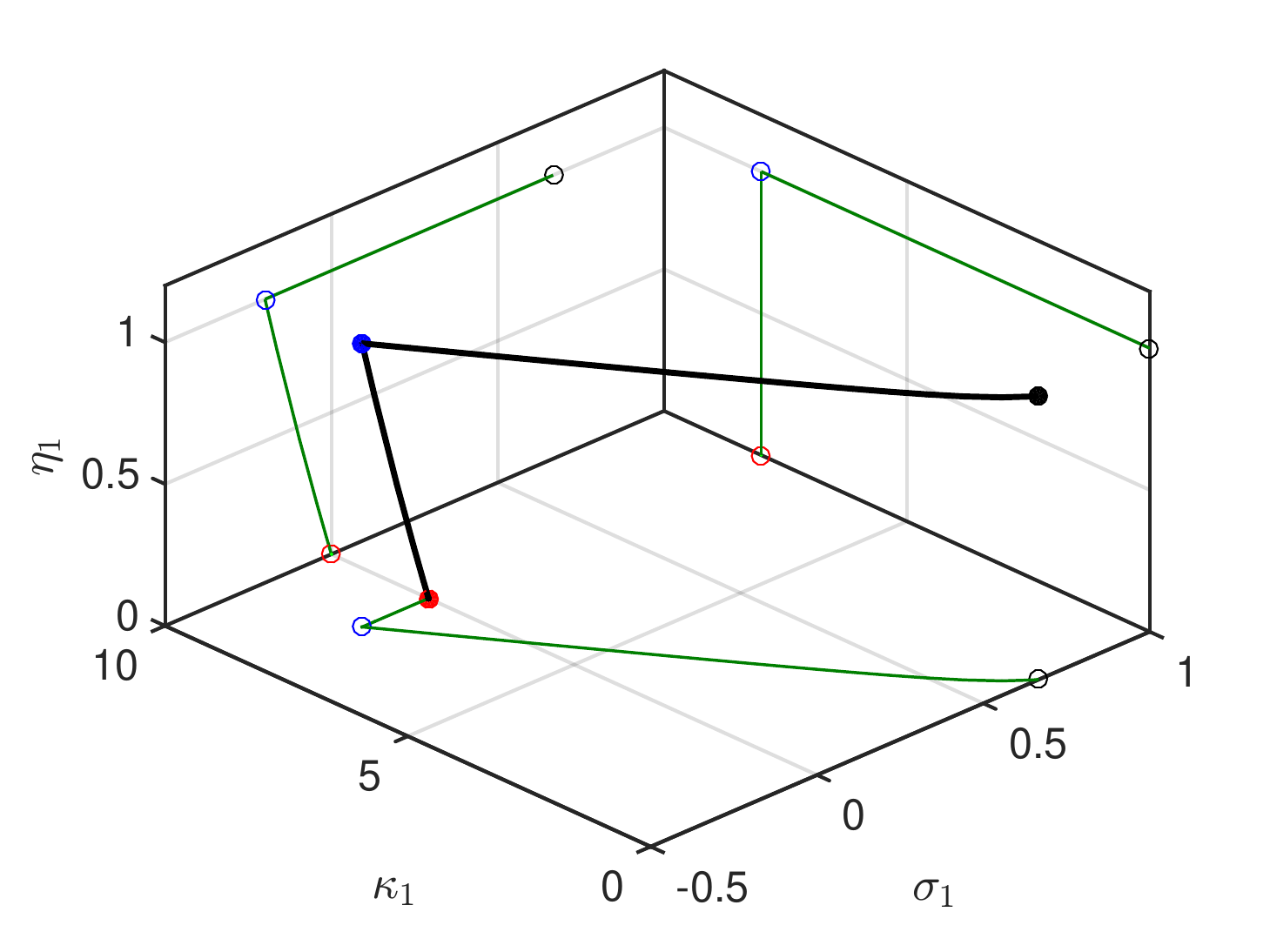}
\caption{Projections of continuation paths associated with a {successive} search for stationary solutions. {Here, dark green thin lines and hollow markers are used to denote projections of black thick lines and filled markers in three-dimensional space onto the three coordinate planes. Gray planes are used to represent tight constraints.} Starting from $u_0=(3,1)$ and holding $\kappa_1$ and $\kappa_2$ fixed at $8$ and $0$, respectively, a fold point in $\mu_1$, denoted by the red dots, is detected along the first solution manifold in the $\eta_1=\sigma_1=\sigma_2=0$ subspace. Along the secondary branch, blue dots denote locations where $\eta_1=1$. With fixed $\eta_1$, the terminal points (black dots) on the tertiary manifolds denote the stationary points where $\kappa_{1}=0$. The magenta square in the left panel corresponds to a singular point for the initial stage of continuation.}
\label{fig:cont-path-L1}
\end{figure}

Notably, the initial continuation from $u_0$ terminates with a failure of the Newton solver to converge at the singular point $(x,y,\mu_1,\eta_1,\sigma_1,\sigma_2)=(7/5, 7/5,17/25,0,0,0)$ on $G_2=0$ represented by the magenta-colored dot in the left panel. It is easy to show that a second branch of solutions, confined to the $G_2=0$ surface and with varying $\sigma_1$ and $\sigma_2$, also terminates at this point. This branch extends in a direction oppositely aligned (angle greater than $90^\circ$) to the direction along which continuation arrives at the singular point, thereby explaining the failure of the pseudo-arclength algorithm to bypass the singular point {(cf.~the right panel of~\cref{fig:arclength})}. Moreover, since $\eta$ decreases from $0$ as $\sigma_2$ increases from $0$ along this secondary branch, it cannot be used to reach a point with $\eta=1$.

\begin{figure}[ht!]
\centering
\includegraphics[width=0.85\textwidth]{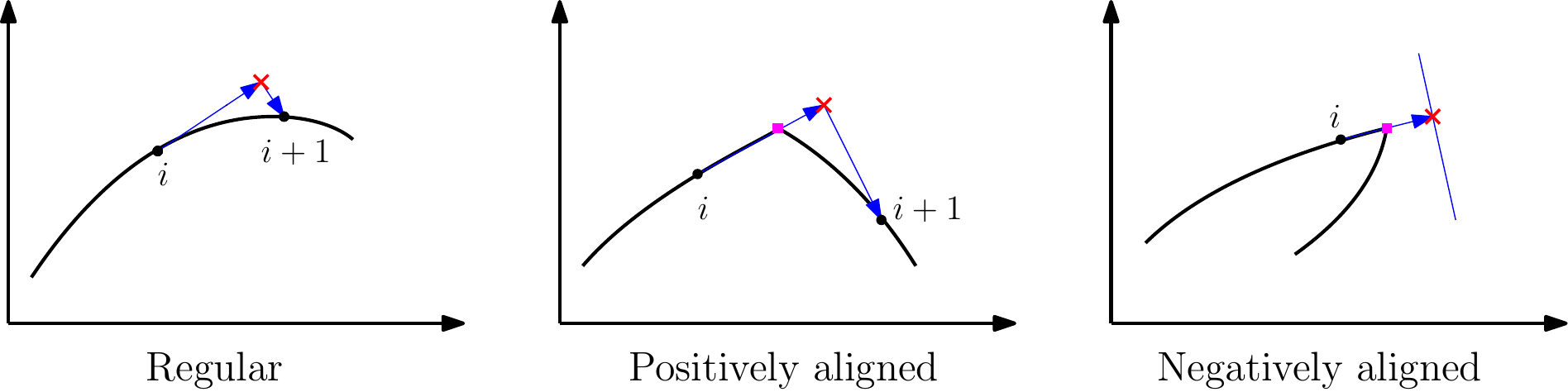}
\caption{{Illustration of the pseudo-arclength algorithm for one-dimensional continuation. Starting from a point $i$ located on a solution manifold, the next point $i+1$ is obtained by two steps in the algorithm. In the first step, a predictor denoted by a red cross is generated along the tangent direction at point $i$. A projection condition is then applied in the second step to locate the point $i+1$ on the manifold. Singular points are denoted by magenta squares. In the left panel, no singular point is encountered. In the middle panel, the extending directions of two branches terminating on the singular point are positively aligned and the algorithm may be able to bypass such a singular point. By contrast, in the right panel, the secondary branch extends in a direction oppositely aligned to the first branch, and no local solution exists to the projection condition. The algorithm is not able to bypass the singular point in this case.}}
\label{fig:arclength}
\end{figure}

As an alternative, suppose that $u_0\in U_{-/+}$ and let $\kappa_{0,2}=\chi(0,-G_2({u_0}))=2G_2(u_0)>0$. It follows that $u_0$ lies on a one-dimensional solution manifold of the equation
\begin{equation}\label{exred2}
\chi(0,-G_2({u}))-\kappa_{0,2}=0\quad\Longleftrightarrow\quad G_2(u)=G_2(u_0).
\end{equation}
The point
\begin{equation}
u_1=\left( \frac{4-2G_2(u_0)}{3}, \frac{4+G_2(u_0)}{3}\right)
\end{equation}
is a stationary point of $\Psi_1$ on this manifold, but $G_1(u_1)>0$. The segment of the manifold between $u_0$ and $u_1$ intersects the $G_1=0$ surface at
\begin{equation}
u_2=\left(\frac{3-4G_2(u_0)}{5},\frac{3+G_2(u_0)}{5}\right).
\end{equation}

We again considering the system of equations in \eqref{exrest}, this time with $\kappa_1=0$, $\kappa_2=\kappa_{0,2}$, and unknowns $(u,\mu_1,\eta_1,\sigma_1,\sigma_2)$. As before, the one-dimensional solution manifold of \eqref{exred2} through $u_0$ corresponds to a locally unique one-dimensional solution manifold of \eqref{exrest} through $(u_0,\Psi_1(u_0),0,0,0)$. Rather than reaching the stationary point of $\Psi_1$, this manifold terminates at $u_2$, which is a singular point for the corresponding Jacobian. Interestingly, a secondary branch, locally parameterized by $\eta$ and along which $G_1\equiv 0$ and $\sigma_1,\sigma_2\ne 0$, also terminates at this point. Substitution yields $x=3-4y$, $\sigma_1=2(1+2y)\eta/5$, and $\sigma_2=12(1-3y)\eta/5$, where $y$ is implicitly defined by
\begin{equation}
\chi(12(1-3y)\eta/5,3-5y)-\kappa_{0,2}=0.
\end{equation}
It follows from \eqref{chiform} that $\sigma_1>0$ along this manifold for $\eta\in(0,1]$. Denote the corresponding $u$ for $\eta=1$ by $u_3$.  Then $(u_3,\Psi_1(u_3), 2(1+2y_3)/5,12(1-3y_3)/5,\kappa_{0,2})$ is a solution to \eqref{exrest} with $\eta=1$, $\kappa_1=0$, and unknowns $(u,\mu_1,\sigma_1,\sigma_2,\kappa_2)$.  Indeed, this point lies on a locally unique one-dimensional manifold of such solutions on $G_1=0$, along which $x=3-4y$, $\sigma_1=2(1+2y)/5$, and $\sigma_2=12(1-3y)/5$, where $y$ is implicitly defined by
\begin{equation}
\label{exrest20}
\chi(12(1-3y)/5,3-5y)-\kappa_2=0.
\end{equation}
It is again easy to show from \eqref{chiform} that $\sigma_1>0$ along this manifold for $\kappa_2\in[0,\kappa_{0,2}]$. The corresponding $u$ for $\kappa_2=0$ then equals $\hat{u}$, as expected.

Continuation results are again consistent with the above analysis. Starting from $u_0=(-4,0)\in U_{-/+}$, we have $G_2(u_0)=4$ and $\kappa_{0,2}=8$. Continuation along the solution manifold to \eqref{exrest} with $\kappa_1=0$ and $\kappa_2=\kappa_{0,2}$ results in a curve that appears to terminate at $(x,y,\mu_1,\eta_1,\sigma_1,\sigma_2)=(-2.6000, 1.4000,21.4800,0,0,0)$ represented by the magenta dots in~\cref{fig:cont-path-L2}. As it happens, the pseudo-arclength algorithm manages to cross this point {(cf.~the middle panel in \cref{fig:arclength})} and continuation proceeds along the secondary branch in the $G_1=0$ surface, terminating at the point $(x,y,\mu_1,\eta_1,\sigma_1,\sigma_2)=(-0.2262, 0.8065, 5.0308,1,  1.0452, -3.4071)$ represented by the blue dots in~\cref{fig:cont-path-L2}. Finally, continuation along the solution manifold to \eqref{exrest} with $\eta=1$ and $\kappa_1=0$ yields the curves in~\cref{fig:cont-path-L2} connecting the blue dots to the black dots at $(x,y,\mu_1,\sigma_1,\sigma_2,\kappa_2)=(1.6667,0.3333,1.0000,0.6667,0,0)$, consistent with the analytical solution.

\begin{figure}[h]
\centering
\includegraphics[width=3.0in]{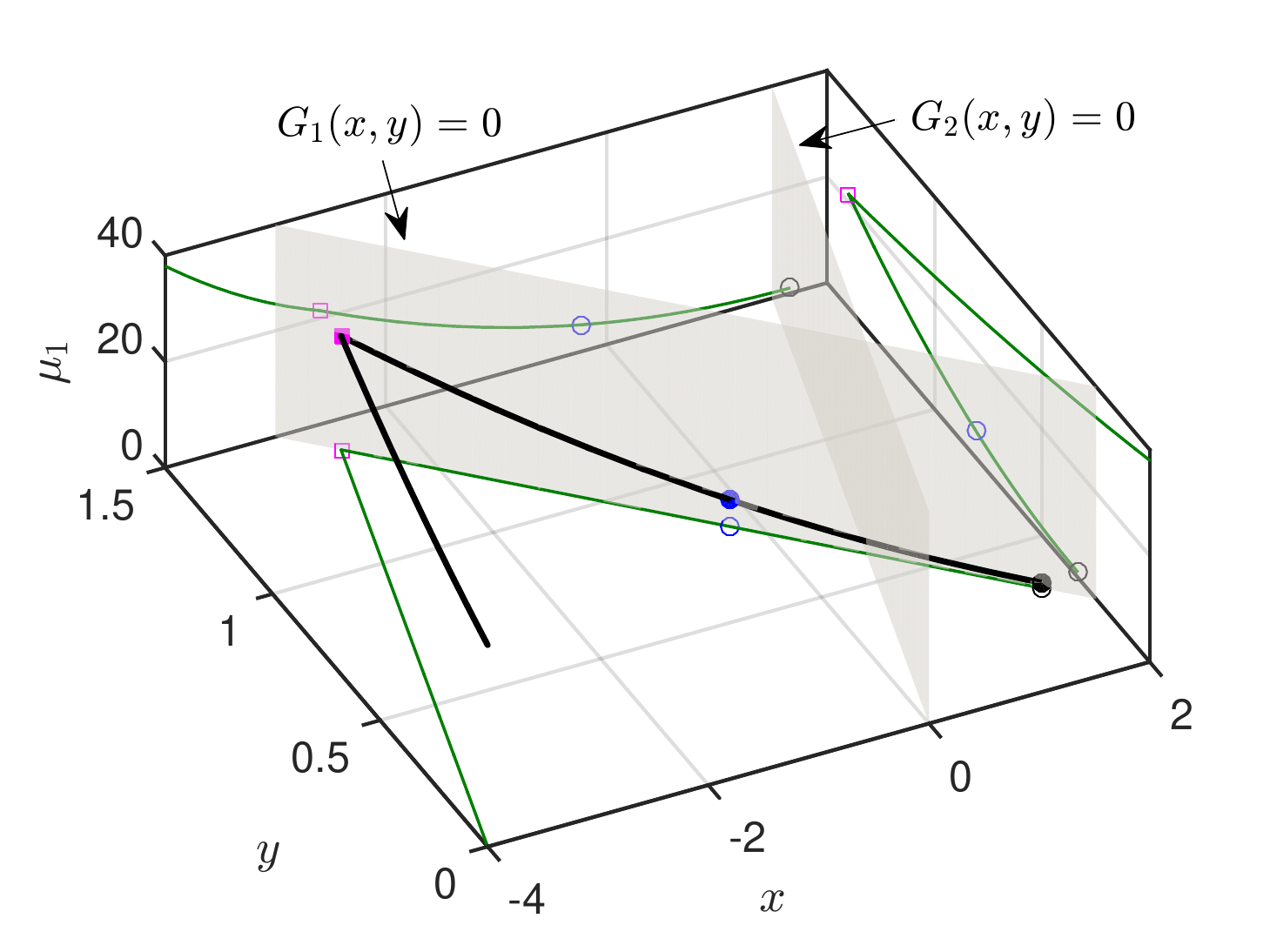}
\includegraphics[width=3.0in]{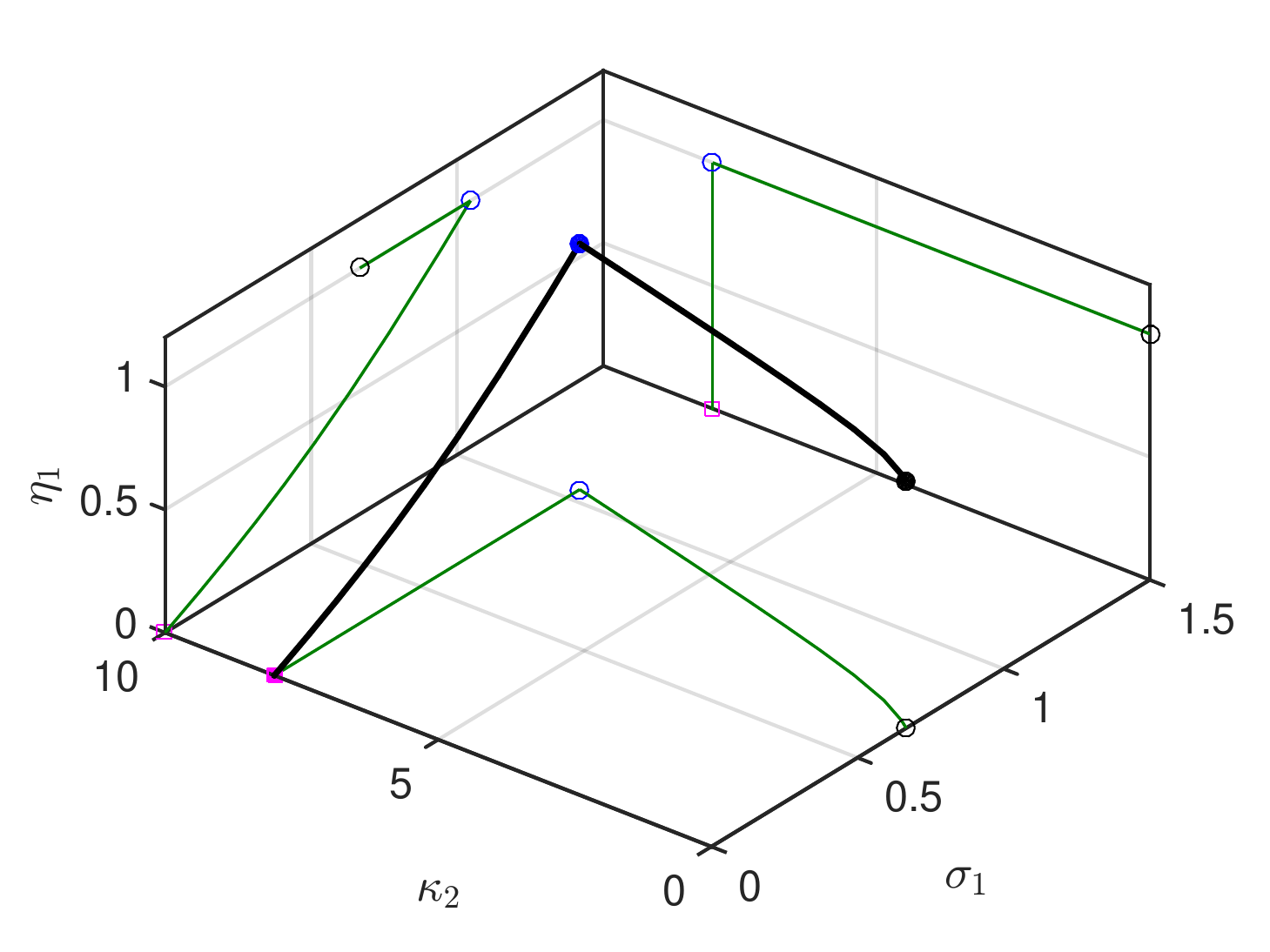}
\caption{Projections of continuation paths  associated with a {successive} search for stationary solutions. {Here, dark green thin lines and hollow markers are used to denote projections of black thick lines and filled markers in three-dimensional space onto the three coordinate planes. Gray planes are used to represent tight constraints.} Starting from $u_0=(-4,0)$ and holding $\kappa_1$ and $\kappa_2$ fixed at $0$ and $8$, respectively, continuation first proceeds along a  solution manifold in the $\eta_1=\sigma_1=\sigma_2=0$ subspace. Rather than terminating at a singular point (magenta squares) on the $G_1=0$ surface, the pseudo-arclength continuation algorithm bypasses this point and switches to a secondary branch in the $G_1=0$ surface along which $\eta_1$, $\sigma_1$, and $\sigma_2$ vary. The first run terminates at the point corresponding to $\eta_1=1$ (blue dots).  In the second run, with fixed $\eta_1$, the terminal points (black dots) on the second manifolds denote the stationary points where $\kappa_{2}=0$.}
\label{fig:cont-path-L2}
\end{figure}

Suppose, instead, that $u_0\in U_{+/+}$ and let $\kappa_{0,1}:=\chi(0,G_1(u_0))=2G_1(u_0)>0$ and $\kappa_{0,2}:=\chi(0,G_2(u_0))=2G_2(u_0)>0$.  We again consider the system of equations in \eqref{exrest}, this time with $\kappa_1=\kappa_{0,1}$, $\kappa_2=\kappa_{0,2}$, and unknowns $(u,\mu_1,\eta_1,\sigma_1,\sigma_2)$. We obtain a one-dimensional solution manifold through $u_0$ along which $\sigma_1=2(4-x-2y)\eta/5$ and $\sigma_2=4(2x-y-3)\eta/5$, where $x$ and $y$ are implicitly defined by
\begin{gather}
\chi(2(4-x-2y)\eta/5,3-x-4y)-\kappa_{0,1}=0,\\
\chi(4(2x-y-3)\eta/5,x-y)-\kappa_{0,2}=0.
\end{gather}
Denote the $u$ corresponding to $\eta=1$ by $u_1$. Then $(u_1,\Psi(u_1),2(4-x_1-2y_1)/5,4(2x_1-y_1-3)/5,\kappa_{0,1})$ is a solution to \eqref{exrest} with $\eta=1$, $\kappa_2=\kappa_{0,2}$, and unknowns $(u,\mu_1,\sigma_1,\sigma_2,\kappa_1)$. Indeed, this point lies on a locally unique one-dimensional manifold of such solutions, along which $\sigma_1=2(4-x-2y)/5$ and $\sigma_2=4(2x-y-3)/5$, where $x$ and $y$ are implicitly defined by
\begin{gather}
\chi(2(4-x-2y)/5,3-x-4y)-\kappa_1=0,\\
\chi(4(2x-y-3)/5,x-y)-\kappa_{0,2}=0.
\end{gather}
Denote the $u$ {corresponding} to $\kappa_1=0$ by $u_2$. Then $(u_2,\Psi_1(u_2),2(4-x_2-2y_2)/5,4(2x_2-y_2-3)/5,\kappa_{0,2})$ is a solution to \eqref{exrest} with $\eta=1$, $\kappa_1=0$, and unknowns $(u,\mu_1,\sigma_1,\sigma_2,\kappa_2)$. Indeed, this point lies on a locally unique one-dimensional manifold of such solutions, along which $\sigma_1=2(4-x-2y)/5$ and $\sigma_2=4(2x-y-3)/5$, where $x$ and $y$ are implicitly defined by
\begin{gather}
\chi(2(4-x-2y)/5,3-x-4y)=0,\\
\chi(4(2x-y-3)/5,x-y)-\kappa_2=0.
\end{gather}
The corresponding $u$ for $\kappa_2=0$ then equals $\hat{u}$, as expected.

The numerical results in~\cref{fig:cont-path-L12} confirm these predictions. Here, $u_0=(1,2)$ implies that $G_1(u_0)=6$, $G_2(u_0)=1$, $\kappa_{0,1}=12$, and $\kappa_{0,2}=2$. Continuation along the solution manifold to \eqref{exrest} with $\kappa_1=\kappa_{0,1}$ and $\kappa_2=\kappa_{0,2}$ results in a curve that intersects the point $(x,y,\mu_1,\eta_1,\sigma_1,\sigma_2)=(1.7487,1.7479,1.1819,1.0000,-0.4978,-1.0004)$ represented by the yellow dots in~\cref{fig:cont-path-L12}. Continuation from this point along the solution manifold to \eqref{exrest} with $\eta=1$ and $\kappa_2=\kappa_{0,2}$ results in a curve that intersects the point $(x,y,\mu_1,\sigma_1,\sigma_2,\kappa_1)=(1.0000,0.5000,1.5000,0.8000,-1.2000,0)$ represented by the blue dots~\cref{fig:cont-path-L12}. Finally, consistent with the analytical solution, continuation along the solution manifold to \eqref{exrest} with $\eta=1$ and $\kappa_1=0$ yields the curves in~\cref{fig:cont-path-L12} connecting the blue dots to the black dots at $(x,y,\mu_1,\sigma_1,\sigma_2,\kappa_2)=(1.6667,0.3333,1.0000,0.6667,0,0)$.

\begin{figure}[h]
\centering
\includegraphics[width=3.0in]{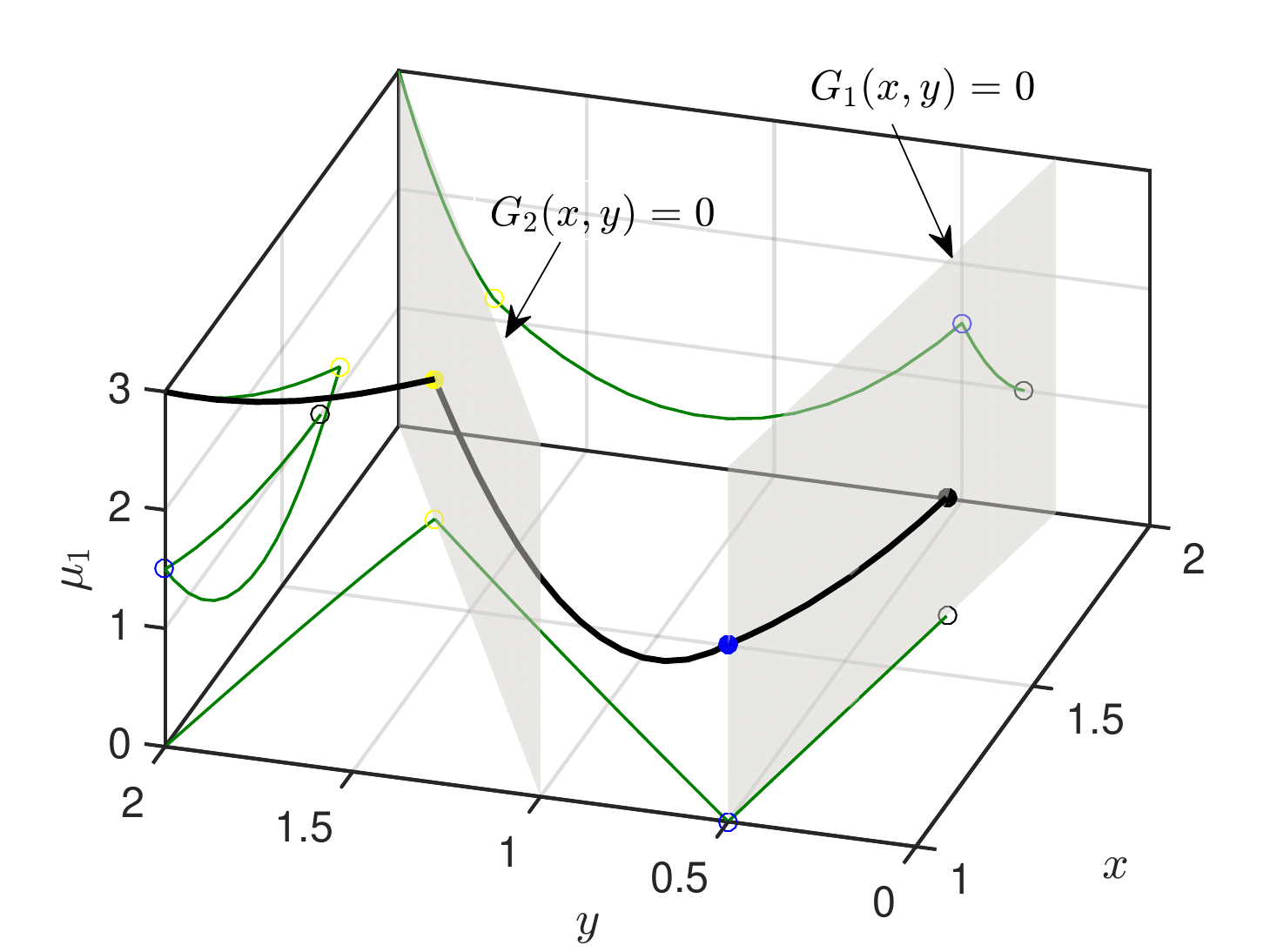}
\includegraphics[width=3.0in]{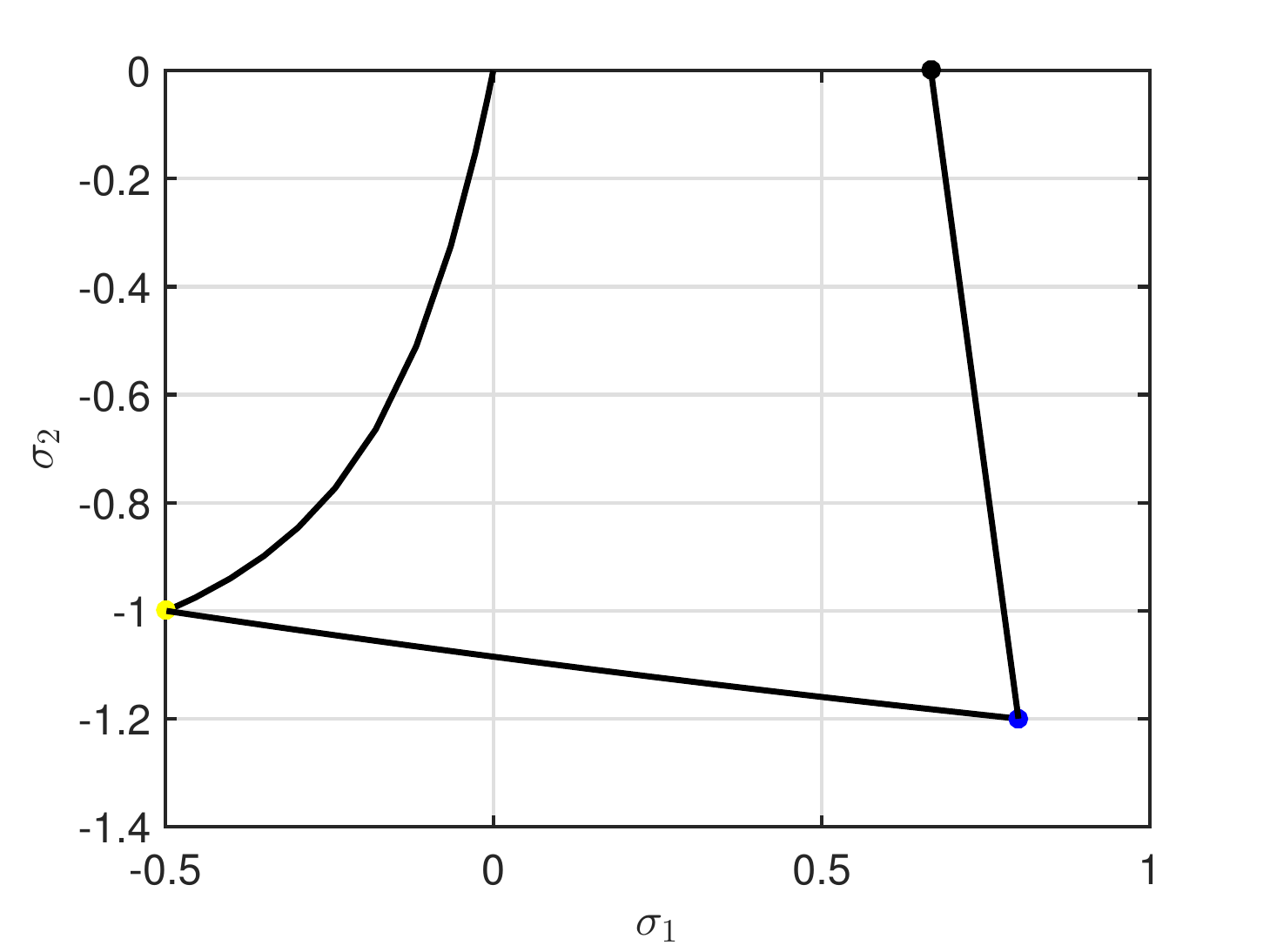}
\caption{Projections of continuation paths associated with a {successive} search for stationary solutions. {Here, dark green thin lines and hollow markers are used to denote projections of black thick lines and filled markers in three-dimensional space onto the three coordinate planes. Gray planes are used to represent tight constraints.} Starting from $u_0=(1,2)$ and holding $\kappa_1$ and $\kappa_2$ fixed at $12$ and $2$, respectively, continuation proceeds along a  one-dimensional solution manifold until $\eta_1=1$ (yellow dots). Fixing $\eta_1$ and varying $\kappa_1$, continuation is conducted until $\kappa_1=0$ (blue dots). Finally, with fixed $\kappa_1$ and $\kappa_2$ free to vary, the terminal points (black dots) on the tertiary manifolds denote the stationary points where $\kappa_{2}=0$.}
\label{fig:cont-path-L12}
\end{figure}

We finally consider $u_0\in U_{-/-}$. Locally, the solutions to the system of equations \eqref{exrest} with $\kappa_1=0$, $\kappa_2=0$, and unknowns $(u,\mu_1,\eta_1,\sigma_1,\sigma_2)$ constitute a two-dimensional manifold of the form $(u,\Psi(u),0,0,0)$ for arbitrary $u\approx u_0$. We obtain a one-dimensional manifold by introducing the function $\Psi_2:u\mapsto y$ and considering the system of equations
\begin{gather}\label{exrest2}
\left\{\begin{array}{c}
\Psi_1(u)-\mu_1=0,\\
\Psi_2(u)-\mu_2=0,\\
\left(D\Psi(u)\right)^\top\eta+\left(DG(u)\right)^\top\sigma=0,\\
\chi\left(\sigma_1,-G_1(u)\right)-\kappa_1=0,\\
\chi\left(\sigma_2,-G_2(u)\right)-\kappa_2=0,\end{array}\right.
\end{gather}
with $\kappa_1=0$, $\kappa_2=0$, $\mu_2=y_0$, and unknowns $(u,\mu_1,\eta_1,\eta_2,\sigma_1,\sigma_2)$. Along this manifold, $\eta_1=\eta_2=\sigma_1=\sigma_2=0$ while $y=y_0$ and $x$ ranges between $3-4y_0$ and $y_0$ corresponding to singular points on the $G_1=0$ and $G_2=0$ surfaces. The point with $u_1=(2,y_0)$ is a stationary point of $\Psi_1$ along this manifold that lies in $U_{-/-}$ provided that $y_0<1/4$. The stationary point corresponds to a branch point through which runs a secondary one-dimensional manifold with solutions of the form $(2,y_0,2(y_0-1)^2,\eta_1,4(1-y_0)\eta_1,0,0)$. It follows that $(2,y_0,2(y_0-1)^2,y_0,4(1-y_0),0,0)$ is a solution to the system of equations \eqref{exrest} with $\kappa_1=\kappa_2=0$, $\eta_1=1$, and unknowns $(u,\mu_1,\mu_2,\eta_2,\sigma_1,\sigma_2)$. The corresponding one-dimensional solution manifold consists of solutions of the form $(2,\mu_2,2(\mu_2-1)^2,\mu_2,4(1-\mu_2),0,0)$ and terminates at the singular point $(2,1/4,9/8,1/4,3,0,0)$ on the $G_1=0$ surface. As before, a secondary one-dimensional manifold, along which $G_1\equiv 0$ and $\sigma_1\ne 0$, also terminates at this point. Substitution yields $x=3-4y$, $\sigma_1=2(4y-1)$, $\sigma_2=0$, and $\eta_2=12(1-3y)$. The corresponding $u$ for $\eta_2=0$ then equals $\hat{u}$, as expected. These predictions are confirmed by the numerical results in~\cref{fig:cont-path-Lempty} where $y_0=-2$. The singular point on $G_1=0$ is again fortuitously bypassed by the pseudo-arclength algorithm {(cf.~the middle panel in \cref{fig:arclength})}.

\begin{figure}[ht!]
\centering
\includegraphics[width=3.0in]{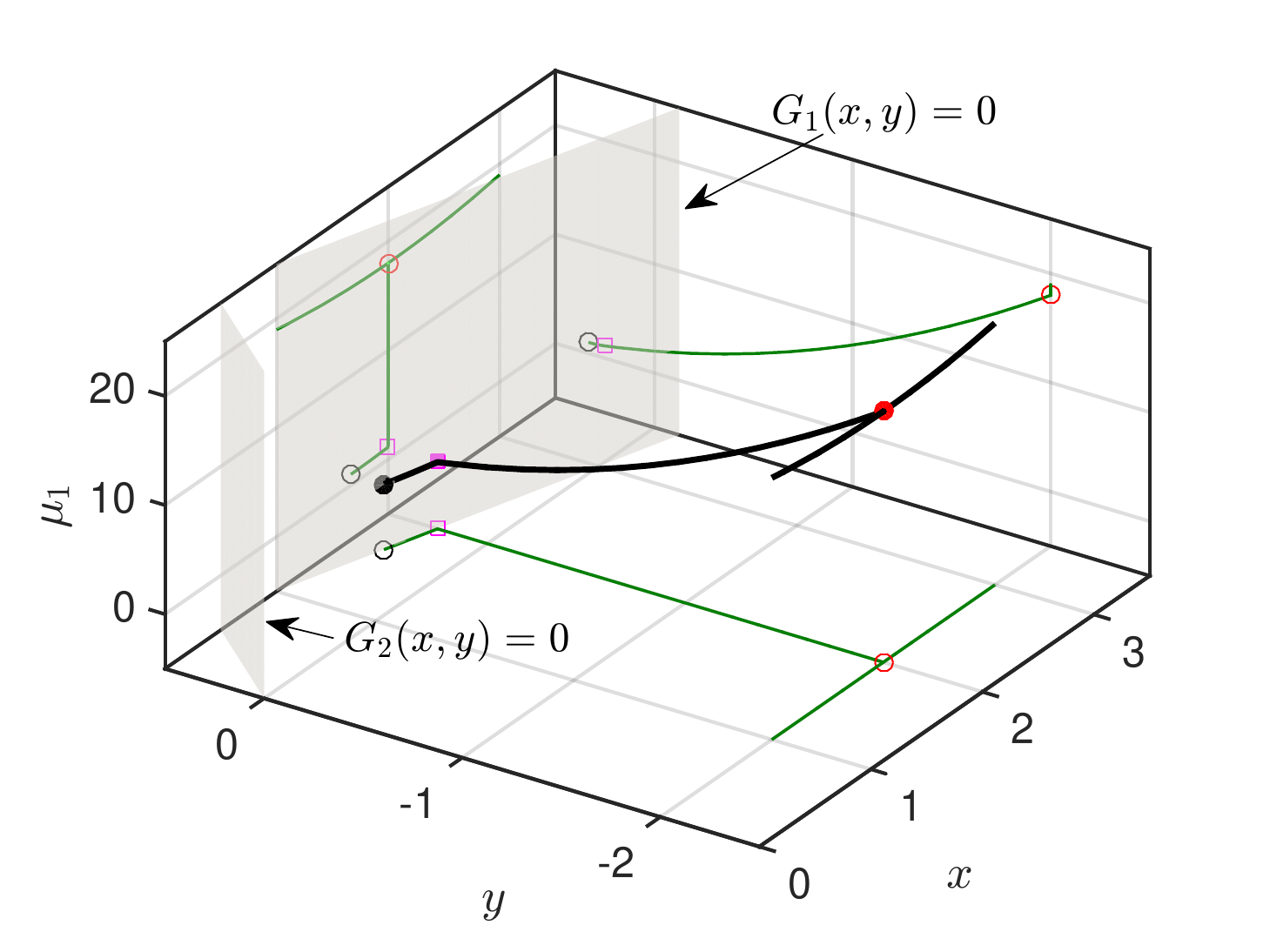}
\includegraphics[width=3.0in]{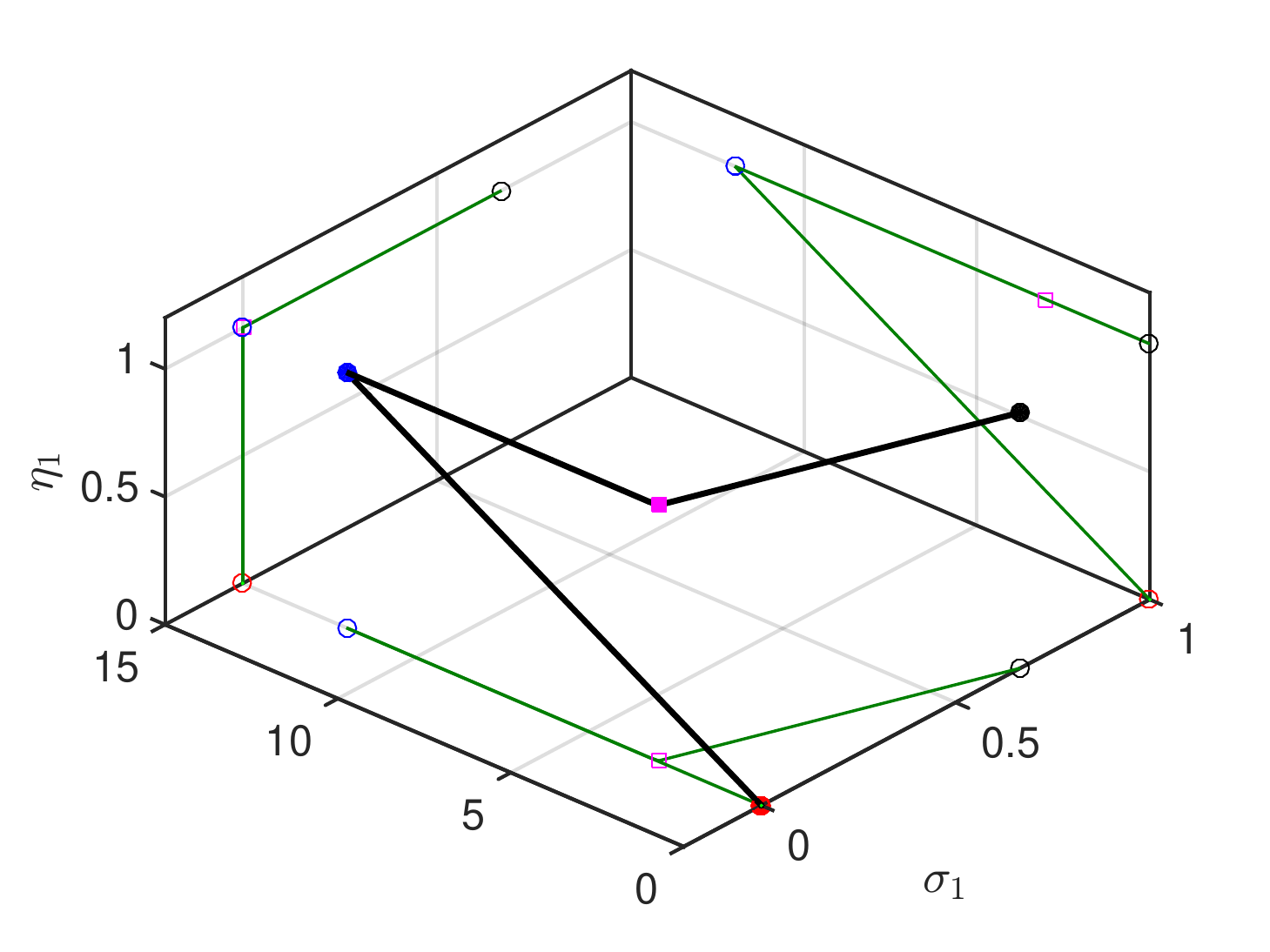}
\caption{Projections of continuation paths associated with a {successive} search for stationary solutions. {Here, dark green thin lines and hollow markers are used to denote projections of black thick lines and filled markers in three-dimensional space onto the three coordinate planes. Gray planes are used to represent tight constraints.} Starting from $u_0=(1,-2)$ and holding $\kappa_1$, $\kappa_2$ , and $\mu_2$ fixed at $0$, $0$, and $-2$, respectively, a fold point  in $\mu_1$, denoted by the red dots, is detected along the first solution manifold in the $\eta_1=\eta_2=\sigma_1=\sigma_2=0$ subspace. Along the secondary branch, blue dots denote locations where $\eta_1=1$. Notably, $u$ remains unchanged along this branch. Finally, with fixed $\eta_{1}$ and $\eta_2$ free to vary, the terminal points (black dots) on the tertiary manifolds denote the stationary points where $\eta_2=0$. This run bypasses a singular point (magenta squares) and then continues in the plane $G_1(x,y)=0$.}
\label{fig:cont-path-Lempty}
\end{figure}

%\section{Adjoint formulation}
%\label{sec:adjt-form}
%
%\section{Complementarity problems}
%\label{sec:cp}

\section{Successive continuation}
\label{sec:cont}

The finite-dimensional motivating example in the previous section highlights a general approach to locating candidate stationary points of an objective function along a constraint manifold. In this section, we generalize this procedure to the infinite-dimensional context that includes boundary-value problem and integral constraints and objective functions.

We proceed to consider the function
\begin{equation}
\label{eq:aug}
F_\mathrm{aug}:(u,\lambda,\eta,\sigma,\mu,\nu,\kappa)\mapsto\begin{pmatrix}\Phi(u)\\\Psi(u)-\mu\\(D\Phi({u}))^\ast{\lambda}+(D\Psi({u}))^\ast{\eta}+(DG({u}))^\ast{\sigma}\\ {\eta}-\nu\\ K({\sigma},-G({u}))-\kappa \end{pmatrix},
\end{equation}
which augments the function $F$ by incorporating a subset of the left-hand sides of the necessary KKT conditions. Various restrictions of $F_{\mathrm{aug}}$ result by fixing different subsets of the components of $\mu$, $\nu$ and $\kappa$. For example, the preceding discussion shows that $\big(u,\lambda,\eta,\sigma,\mu\big)=\big(\hat{u}, \hat{\lambda},\hat{\eta},\hat{\sigma},\Psi(\hat{u})\big)$ is a root of the restriction obtained by fixing $\nu_1=1$, $\nu_{\{2,\ldots,l\}}=0$ and $\kappa=0$. It is notably difficult to locate such a root without an \emph{a priori} approximation. To overcome this difficulty, we propose a modification to the successive continuation algorithm introduced by Kern{\'e}vez and Doedel~\cite{DKK91} and described by us in~\cite{staged_adjoint} in the context of a nonlinear function similar in form to \eqref{eq:aug} but with $q=0$.

Specifically, suppose that $u_0$ is a root of $\Phi$ for which the set of active constraints is empty, i.e., $\{i:G_i(u_0)=0\}=\emptyset$, thus avoiding the singularity of $\chi$ at $(0,0)$. Then, $(u_0,0,0,0,\Psi(u_0),0,\kappa_0)$ is a root of $F_\mathrm{aug}$ provided that the elements of $\kappa_0$ indexed by $\mathbb{Z}:=\{i:G_i(u_0)< 0\}$ equal $0$, while those indexed by $\mathbb{P}:=\{i:G_i(u_0)> 0\}$ are positive. We proceed to develop a series of continuation problems whose solutions correspond to points on a sequence of embedded manifolds that continuously connect this initial root of the augmented continuation problem to a root of this problem with $\nu_1=1$, $\nu_{\{2,\ldots,l\}}=0$, and $\kappa=0$.

To this end, choose some index sets $\mathbb{I}\subset \{2,\ldots,l\}$ and $\mathbb{J}:=\{2,\ldots,l\}\setminus \mathbb{I}$, and consider the restriction $F_{\mathrm{rest}}$ obtained by fixing the values of $\mu_{\mathbb{I}}$, $\nu_{\mathbb{J}}$, and $\kappa$ to $\Psi_{\mathbb{I}}(u_0)$, $0$, and $\kappa_0$, respectively. It follows by construction that $(u,\lambda,\eta,\sigma,\mu_1,\mu_\mathbb{J},\nu_1,\nu_\mathbb{I})=(u_0,0,0,0,\Psi_1(u_0),\Psi_{\mathbb{J}}(u_0),0,0)$ is a root of $F_\mathrm{rest}$. Indeed, by continuity of $G$, every root $u\approx u_0$ of the function
\begin{equation}
\label{eq:red}
F_\mathrm{red}:u\mapsto\begin{pmatrix}\Phi(u)\\\Psi_{\mathbb{I}}(u)-\Psi_{\mathbb{I}}(u_0)\\ K_{\mathbb{P}}(0,-G(u))-\kappa_{0,\mathbb{P}}\end{pmatrix}
\end{equation}
corresponds to a root of the form $(u,0,0,0,\Psi_1(u),\Psi_\mathbb{J}(u),0,0)$ of $F_\mathrm{rest}$. 

Now suppose that the range of $DF_\mathrm{red}(u_0)$ is $Y\times\mathbb{R}^{|\mathbb{I}|}\times\mathbb{R}^{|\mathbb{P}|}$ and that its nullspace is of dimension $d-|\mathbb{I}|-|\mathbb{P}|$, where $d$ equals the dimension of the nullspace of $D\Phi(u_0)$. It follows from the implicit function theorem that all roots of $F_\mathrm{red}$ near $u_0$ lie on a locally unique $d-|\mathbb{I}|-|\mathbb{P}|$-dimensional manifold. Consider the special case that $|\mathbb{I}|+|\mathbb{P}|=d-1$. Then, by Corollary~\ref{cor1}, all roots of $F_\mathrm{rest}$ sufficiently close to $(u_0,0,0,0,\Psi_1(u_0),\Psi_{\mathbb{J}}(u_0),0,0)$ lie on a one-dimensional manifold of points of the form $(u,0,0,0,\Psi_1(u),\Psi_{\mathbb{J}}(u),0,0)$, for some root $u$ of $F_\mathrm{red}$ provided that $\left(D\Psi_1(u_0)\right)^\ast$ is linearly independent of $\left(DF_\mathrm{red}(u_0)\right)^*$. If, instead, $u_0$ is a stationary point of $\Psi_1$ along the one-dimensional solution manifold of $F_\mathrm{red}=0$, then Lemma~\ref{lem2} implies that $(u_0,0,0,0,\Psi_1(u_0),\Psi_{\mathbb{J}}(u_0),0,0)$ is a branch point of $F_\mathrm{rest}$ through which runs a secondary one-dimensional solution manifold, locally parameterized by $\eta_1$, along which $\lambda$, $\eta_\mathbb{I}$, and $\sigma_\mathbb{P}$ vary.

Suppose that no element of $G_\mathbb{Z}$ equals $0$ along this manifold for $\eta_1\in [0,1]$. Continuation can then proceed from the point with $\eta_1=1$ along a sequence of one-dimensional manifolds of solutions to $F_\mathrm{rest}=0$ obtained by fixing $\eta_1=1$ and successively moving indices from $\mathbb{I}$ to $\mathbb{J}$ and fixing the corresponding elements of $\nu$ once they equal $0$. Suppose that no element of $G_\mathbb{Z}$ equals $0$ along any such segment. Continuation can then proceed from the point with $\eta_1=1$, $\eta_{2,\dots,l}=0$ along a sequence of one-dimensional manifolds of solutions to $F_\mathrm{rest}$ obtained by fixing $\eta$ and successively allowing the elements of $\kappa_\mathbb{P}$ to vary and fixing them once they equal $0$. Provided that no element of $G_\mathbb{Z}$ equals $0$ along any such segment, the final point corresponds to the sought stationary point.

Alternatively, consider the case that $|\mathbb{I}|+|\mathbb{P}|=d$, in which case $DF_\mathrm{red}(u_0)$ is a bijection. Lemma~\ref{lem3} then implies that $(u_0,0,0,0,\Psi_1(u_0),\Psi_{\mathbb{J}}(u_0),0,0)$ lies on a one-dimensional manifold of solutions to $F_\mathrm{rest}=0$ locally parameterized by $\eta_1$, along which $\lambda$, $\eta_\mathbb{I}$, and $\sigma_\mathbb{P}$ vary. Suppose that no element of $G_\mathbb{Z}$ equals $0$ along this manifold for $\eta_1\in [0,1]$. The stationary point $\hat{u}$ may again be sought following the approach in the preceding paragraph.

Finally, note that if any element of $G_\mathbb{Z}$ were to equal $0$ along any of the segments, Lemma~\ref{lem4} allows for the possibility of branch switching to a one-dimensional solution manifold along the corresponding zero surface. This manifold is again locally parameterized by $\eta_1$, and $\sigma_k,\sigma_\mathbb{P}\ne 0$ for $\eta_1$ close, but not equal to $0$. The stationary point $\hat{u}$ may again be sought following the successive continuation approach.

The various approaches to locating a stationary point in the example in Section~\ref{sec:example} correspond to the possibilities discussed above. Throughout the analysis, $d=2$. 
\begin{itemize}
\item In the case that $u_0\in U_{+/-}$, $\mathbb{I}=\emptyset$ and $\mathbb{P}=\{1\}$ so that $|\mathbb{I}|+|\mathbb{P}|=d-1$. The analysis proceeds by locating a fold along the solution manifold with trivial Lagrange multipliers, branch switching to a secondary branch with nontrivial Lagrange multipliers, and then driving $\kappa_1$ to $0$. 
\item In the case that $u_0\in U_{-/+}$, $\mathbb{I}=\emptyset$ and $\mathbb{P}=\{2\}$ so that, again, $|\mathbb{I}|+|\mathbb{P}|=d-1$. The analysis proceeds by continuing to a singular point on $G_1=0$, branch switching onto a secondary branch on $G_1=0$, and then driving $\kappa_2$ to $0$.
\item In the case that $u_0\in U_{+/+}$, $\mathbb{I}=\emptyset$ and $\mathbb{P}=\{1,2\}$ so that $|\mathbb{I}|+|\mathbb{P}|=d$. The analysis proceeds by continuing along a branch of nontrivial Lagrange multipliers and then successively driving $\kappa_1$ and $\kappa_2$ to $0$. 
\item Finally, in the case that $u_0\in U_{-/-}$, the problem is enlarged with the function $\Psi_2$, thereby making $\mathbb{I}=\{2\}$ and $\mathbb{P}=\emptyset$ so that, again, $|\mathbb{I}|+|\mathbb{P}|=d-1$. The analysis proceeds by locating a fold along the solution manifold with trivial Lagrange multipliers, branch switching to a secondary branch with nontrivial Lagrange multipliers, and then driving $\eta_2$ to $0$, first along a branch with $G_1\ne 0$ to a singular point on $G_1=0$ and then branch switching onto a secondary branch on $G_1=0$.
\end{itemize}
As we see from this enumeration, the different scenarios may not be identified a priori, but suggest a great degree of flexibility to the analyst when choosing the initial point $u_0$ and the set $\mathbb{I}$.

{We conclude this section with a few comments on the proposed algorithm:}
\begin{itemize}
\item {\emph{Initialization.} The algorithm requires the user to select the initial point $u_0$ and the set of initially inactive continuation parameters $\mathbb{I}$ such that $\mathbb{I}+\mathbb{P}\in\{d,d-1\}$, where $d$ denotes the dimension of the solution manifold to the zero problem $\Phi(u)=0$. As seen in the motivating example, there is a great deal of flexibility in this selection, one that may allow for different approaches to one of possibly many local extrema. In particular, initial solution guesses in the infeasible region may be used for the successful search of optima. We are not able to propose a systematic selection algorithm beyond the principles outlined above.}

\item {\emph{Number of continuation runs.} 
The number of subproblems analyzed in the successive continuation approach is determined by the number of control or design variables rather than by the number of constraints. Indeed, in the generalized Kern{\'e}vez and Doedel approach, $d+1$ successive continuation runs are involved to obtain optimal solutions. Specifically, if $\mathbb{I}+\mathbb{P}=d$, continuation is performed until $\nu_1=1$ in the first run and the remaining $d$ runs drive nonzero elements of $\nu_{\mathbb{I}}$ and $\kappa_{\mathbb{P}}$ to $0$, one at a time. If, instead, $\mathbb{I}+\mathbb{P}=d-1$, due to branch switching, the first \emph{two} runs are conducted to drive $\nu_1$ to $1$. The remaining $d-1$ run are then successively performed to drive $\nu_{\mathbb{I}}$ and $\kappa_{\mathbb{P}}$ to $0$.}

\item {\emph{Continuation order.} The algorithm requires the user to commit to an order in which elements associated with $\mathbb{I}$ are released, and constraints associated with $\mathbb{P}$ are imposed. If the solution to the first-order necessary conditions is not unique, different choices may yield distinct locally optimal solutions, as illustrated in ref.~\cite{staged_adjoint}. We leave the effects of the imposed order of constraints to future studies.}
\end{itemize}

\section{Some implementation details}
\label{sec:implementation}
{A fundamental element of the successive continuation algorithm is the analysis of solutions of the augmented continuation problem $F_{\mathrm{aug}}=0$ with appropriate restrictions on the elements of $\mu$, $\nu$, and $\kappa$. A practical implementation of the algorithm should then consider two perspectives, namely, that of formulating the corresponding continuation problem and that of solving the formulated problem. As we did in the motivating example, one may manually derive the restricted continuation problem and then solve it analytically using symbolic computation packages. Given the difficulty of finding analytical solutions, numerical continuation arises as a powerful alternative for characterizing the solution manifolds for each of the restricted continuation problems. One can apply packages like \textsc{auto}~\cite{auto} and  \textsc{coco}~\cite{coco} to perform such continuation.}

{The numerical results documented in this paper were produced using \textsc{coco} because of its support for automatic generation of the corresponding adjoints (or discretized approximations thereof). The construction of adjoints is not a trivial task if the optimization problem has differential or integral constraints. Packages supporting the automatic computation of adjoints include \textsc{sundials}~\cite{sundials} for ordinary differential equations and differential-algebraic equations, and \textsc{dolfin-adjoint}~\cite{dolfin-adjoint} for partial differential equations. However, numerical continuation is not available in these packages. A recent release of \textsc{coco} provides a predefined library of realizations for the adjoints of common types of integral, differential, and algebraic operators. This feature of \textsc{coco}, coupled with the ease in which continuation parameters may be fixed or released, makes the implementation of the algorithm both intuitive and straightforward using this package. Nevertheless, every problem considered in this paper could in theory be approached using other continuation packages.}

A key feature of \textsc{coco} is its support for a staged paradigm of problem construction, in which a continuation problem is decomposed into a ``forward-coupled'' set of equations. Specifically, at each stage of construction, equations are added that depend on subsets of variables introduced in previous stages and a new set of variables. The full set of unknowns is therefore not known until the complete problem has been constructed.

In its original form (released in 2013), \textsc{coco} was designed to construct problems of the form $F(u,\mu)=0$, where
\begin{equation}\label{orig}
F:(u,\mu)\mapsto\begin{pmatrix}\Phi(u)\\\Psi(u)-\mu\end{pmatrix}
\end{equation}
in terms of sets of finite-dimensional \textit{zero functions} $\Phi$, \textit{monitor functions} $\Psi$, \textit{continuation variables} $u$, and \textit{continuation parameters} $\mu$. Various restrictions obtained by fixing elements of $\mu$ could then be realized at run-time using the appropriate \textsc{coco} syntax. Each such restriction would correspond to analysis of an embedded submanifold of the solution manifold to the original \textit{zero problem} $\Phi(u)=0$. For infinite-dimensional problems, $F$ would be implemented in terms of suitable discretizations of $u$, $\Phi$, and $\Psi$.

In our recent work~\cite{staged_adjoint}, an expanded definition of the \textsc{coco} construction paradigm allowed for analysis of problems of the form $F(u,\lambda,\eta,\mu,\nu)=0$, where
\begin{equation}\label{expanded}
F:(u,\lambda,\eta,\mu,\nu)\mapsto\begin{pmatrix}\Phi(u)\\\Psi(u)-\mu\\\Lambda_\Phi^\top(u)\cdot\lambda+\Lambda_\Psi^\top(u)\cdot\eta\\\eta-\nu\end{pmatrix}
\end{equation}
in terms of additional sets of finite-dimensional matrix-valued \textit{adjoint functions} $\Lambda_\Phi$ and $\Lambda_\Psi$, \textit{continuation multipliers} $\lambda$ and $\eta$, and continuation parameters $\nu$. The expanded definition supported the use of a successive continuation technique for locating stationary points of one of the monitor functions along the solution manifold to the zero problem $\Phi(u)=0$. In this context, the transposes $\Lambda_\Phi^\top$ and $\Lambda_\Psi^\top$ represented the (discretized if necessary) adjoints of the Frechet derivatives of the functions $\Phi$ and $\Psi$, and $\lambda$ and $\eta$ were the corresponding \textit{Lagrange multipliers}.

Notably, in \eqref{expanded}, the additional entries $\Lambda_\Phi^\top(u)\cdot\lambda+\Lambda_\Psi^\top(u)\cdot\eta$ may be combined with the original zero problem to form an expanded zero problem in $u$, $\lambda$, and $\eta$. Similarly, in terms of the augmented monitor functions $(u,\lambda,\eta)\mapsto(\Psi(u),\eta)$ the second and last entries may be combined into a single term of the form of the bottom entry in \eqref{orig}. Moreover, $\lambda$ and $\eta$ are introduced only once all the original zero and monitor functions have been added, at which point all continuation variables have been defined. Indeed, the corresponding equations can be added automatically by the \textsc{coco} core, rather than manually constructed by a user, provided that the user constructs $\Lambda_\Phi^\top$ and $\Lambda_\Psi^\top$ either concurrently with the addition of the corresponding zero and monitor functions or at the very least before calling the \textsc{coco} core to perform continuation. This expanded functionality was implemented with the November 2017 release of \textsc{coco} and discussed in tutorial documentation included with the release (\url{https://sourceforge.net/projects/cocotools/files/releases/}).

In the context of the proposed treatment of optimization under simultaneous equality and inequality constraints, we consider a further extension of the \textsc{coco} construction paradigm to problems of the form $F(u,\lambda,\eta,\sigma,\mu,\nu,\xi,\kappa)=0$, where
\begin{equation}
\label{furtherexpanded}
F:(u,\lambda,\eta,\sigma,\mu,\nu,\xi,\kappa)\mapsto\begin{pmatrix}\Phi(u)\\\Psi(u)-\mu\\\Lambda_\Phi^\top(u)\cdot \lambda+\Lambda_\Psi^\top(u)\cdot\eta+\Lambda_G^\top(u)\cdot\sigma\\\eta-\nu\\G(u)-\xi\\ K({\sigma},-G({u}))-\kappa \end{pmatrix}
\end{equation}
in terms of additional sets of finite-dimensional \text{inequality functions} $G$, matrix-valued adjoint functions $\Lambda_G$, continuation multipliers $\sigma$, and continuation parameters $\xi$ and $\kappa$. Here, $\Lambda_G^\top$ represents the (discretized if necessary) adjoint of the Frechet derivative of the function $G$ and $\sigma$ are the corresponding Lagrange multipliers. We match this expanded definition against the original \textsc{coco} construction paradigm by combining $\Lambda_\Phi^\top(u)\cdot \lambda+\Lambda_\Psi^\top(u)\cdot\eta+\Lambda_G^\top(u)\cdot\sigma$ with the original zero problem and defining the augmented monitor function $(u,\lambda,\eta,\sigma)\mapsto(\Psi(u),\eta,G(u),K(\sigma,-G(u)))$. As before, $\lambda$, $\eta$, and $\sigma$ are introduced only once all the original zero, monitor, and inequality functions have been added, at which point all continuation variables have been defined. Again, the corresponding equations can be added automatically by the \textsc{coco} core, rather than manually constructed by a user, provided that the user constructs $\Lambda_\Phi^\top$, $\Lambda_\Psi^\top$, and $\Lambda_G^\top$ either concurrently with the addition of the corresponding zero, monitor, and inequality functions or at the very least before calling the \textsc{coco} core to perform continuation. Such a modification to the \textsc{coco} core was implemented to produce the results reported in this paper.

% give a reivew regarding the library of realizations, also mention the implementation of optimal control

% finite-dimensional case, which is straightforward,,,,,,,
% infinite-dimensional case, refer...,,, and the discretization of adjoints also refers...,,,
In the case that $U$ and $Y$ are finite dimensional, the adjoints of the linearizations of $\Phi$, $\Psi$, and $G$ are straightforward to construct since they simply equal the transposes of the corresponding Jacobians. When $U$ and $Y$ are infinite dimensional, the adjoint contributions can be derived using a Lagrangian formalism. In~\cite{staged_adjoint}, a library of realizations of such functions and their adjoints for algebraic and integro-differential boundary-value problems were established. For example, for boundary-value problems defined in terms of ordinary differential equations, the unknown functions together with the corresponding unknown Lagrange multiplier functions were discretized over a finite mesh in the independent variable in terms of continuous, piecewise-polynomial functions. The original differential equations and the corresponding adjoint differential equations were then discretized by requiring that these be satisfied by the functional approximants on a set of collocation nodes. These collocation nodes and the associated quadrature weights were also used to approximate any integral functions of the original unknowns or Lagrange multipliers. Finally, in the implementation in \textsc{coco}, an adaptive mesh algorithm that varies the sizes and number of mesh intervals was implemented to obtain numerical solutions with desirable accuracy within reasonable limits on computational efforts. We utilize this adaptive mesh algorithm in the computations performed in the next section. For further details of the numerical implementation the reader is referred to~\cite{staged_adjoint,coco-recipes}.

We finally remark on the possibility of restricting continuation to a level surface of $G_k$ for some $k$, either by fixing $\xi_k$, or by fixing $\kappa_k$ provided that $\sigma_k$ is constant along the solution manifold if $G_k\ne 0$ or positive if $G_k=0$. As we saw in the finite-dimensional example, we may switch from continuation away from $G_k=0$ with $\kappa_k=0$ to continuation along $G_k=0$ without any change in the problem definition provided that the corresponding tangent directions are positively aligned, allowing the pseudo-arclength continuation algorithm to bypass the singular point with $\sigma_k=0$ {(cf.~the right panel in \cref{fig:arclength})}. In the contour plot for the Fischer-Burmeister function in~\cref{fig:ncp}, such a transition corresponds to switching from continuation along the vertical segment of the zero contour to continuation along the horizontal segment of the zero contour.
{By detecting the singularity and the corresponding change in the problem definition, one may branch on and off the surface $G_k=0$. The current implementation does not consider such detection and, instead, assumes manual switching between branches.}

\section{Applications}
% give an algebraic example, to show the continuation path in both original variables and adjoint variables
\label{sec:application}

\subsection{Doedel's example}
We revisit a two-point boundary-value problem from \textsc{auto}~\cite{Doedel-ii} augmented by an inequality constraint.
Consider the following objective functional
\begin{equation}
    J:=\frac{1}{10}(p_1^2+p_2^2+p_3^2)+\int_0^1 (x_1(t)-1)^2 \mathrm{d}t
\end{equation}
subject to the differential equations
\begin{equation}
    \dot{x}_1 = x_2,\quad \dot{x}_2 = -p_1\exp(x_1+p_2x_1^2+p_3x_1^4)
\end{equation}
and boundary conditions
\begin{equation}
    x_1(0) = 0,\quad x_1(1) = 0.
\end{equation}
There are three local extrema~\cite{Doedel-ii,staged_adjoint}, two of which violate the integral inequality constraint
\begin{equation}
    G_{\mathrm{int}}:=0.5-\int_0^1 x_1(t) \mathrm{d}t\leq0.
\end{equation}
We apply the formalism from previous sections to locate the remaining extremum. Throughout the analysis, we restrict attention to a computational domain defined by $-0.2\leq p_1\leq3.5$, $-0.2\leq p_2\leq1.5$, $-0.2\leq p_3\leq1.0$, and $0\leq J \leq1.5$. 

In the notation of this paper, $u=(x(t),p),$ $\Phi$ represents the boundary-value problem, $d=3$, and $q=1$. Let $\Psi_1:u\mapsto J$, while $\Psi_i:u\mapsto p_{i-1}$ for $i=2,3,4$, such that $l=4$, and denote the corresponding continuation parameters $\mu_J$, $\mu_{p_1}$, $\mu_{p_2}$, and $\mu_{p_3}$. Finally, let $\kappa_\mathrm{int}$ denote the continuation parameter for the NCP condition associated with the integral inequality constraint. We let $\lambda(t)$, $\lambda_\mathrm{bc}$, $\eta_J$, $\eta_{p_1}$, $\eta_{p_2}$, $\eta_{p_3}$, and $\sigma_\mathrm{int}$ denote the corresponding Lagrange multipliers, and let $\nu_J$, $\nu_{p_1}$, $\nu_{p_2}$, and $\nu_{p_3}$ denote the remaining continuation parameters.

Since $l+q\geq d$, the requirement $|\mathbb{I}|+|\mathbb{P}|=d-1$ can be satisfied by suitable selection of the sets $\mathbb{I}$ and $\mathbb{P}$. We consider two cases here, viz., $\mathbb{P}=\{1\}$ with $\mathbb{I}=\{4\}$ and $\mathbb{P}=\emptyset$ with $\mathbb{I}=\{3,4\}$, respectively.

In the first case, let $u_0=(0,0,0.1,0)$, in which case $G_\mathrm{int}(u_0)=0.5$ in violation of the integral inequality constraint. A one-dimensional solution manifold with trivial Lagrange multipliers results by fixing $\mu_{p_3}$, $\nu_{p_1}=0$, $\nu_{p_2}=0$, and $\kappa_\mathrm{int}=1$ and allowing the remaining continuation parameters to vary. As seen in the right panel of~\cref{fig:cont-path-infeasible}, continuation results in the detection of one local extremum in the value of $\mu_J$. As predicted by \cref{lem2}, continuation is now possible along a secondary solution manifold, emanating from the extremum and parameterized by $\nu_J$. In contrast to the case when $\mathbb{P}=\emptyset$, the value of $u$ changes along this manifold, which is consistent with prediction given by \cref{cor2}. As required by the successive continuation paradigm, continuation is performed until $\nu_J=1$. Next, we proceed to fix $\nu_J = 1$ and allow $\mu_{p_3}$ to vary during continuation from this point until $\nu_{p_3}=0$. We arrive at the desired result by fixing $\nu_{p_3}=0$ and allowing $\kappa_{\mathrm{int}}$ to vary during continuation from this point until $\kappa_{\mathrm{int}}=0$. Notably, while we find it possible to drive $\nu_J$ and $\nu_{p_3}$ monotonically to $1$ and $0$, respectively, in the corresponding continuation runs, two fold points  in the value of $\kappa_{\mathrm{int}}$ are encountered on the way to $0$ in the final continuation run.

% describe the process of continuation and showing the continuation path
% mimic figure_
% two different ways regarding the selection of initial guesses

%Here we restricted attention to  variations of $\nu_{p_2}$ and $\nu_{p_3}$ that approach 0 monotonically.

\begin{figure}[h]
\centering
\includegraphics[width=3.0in]{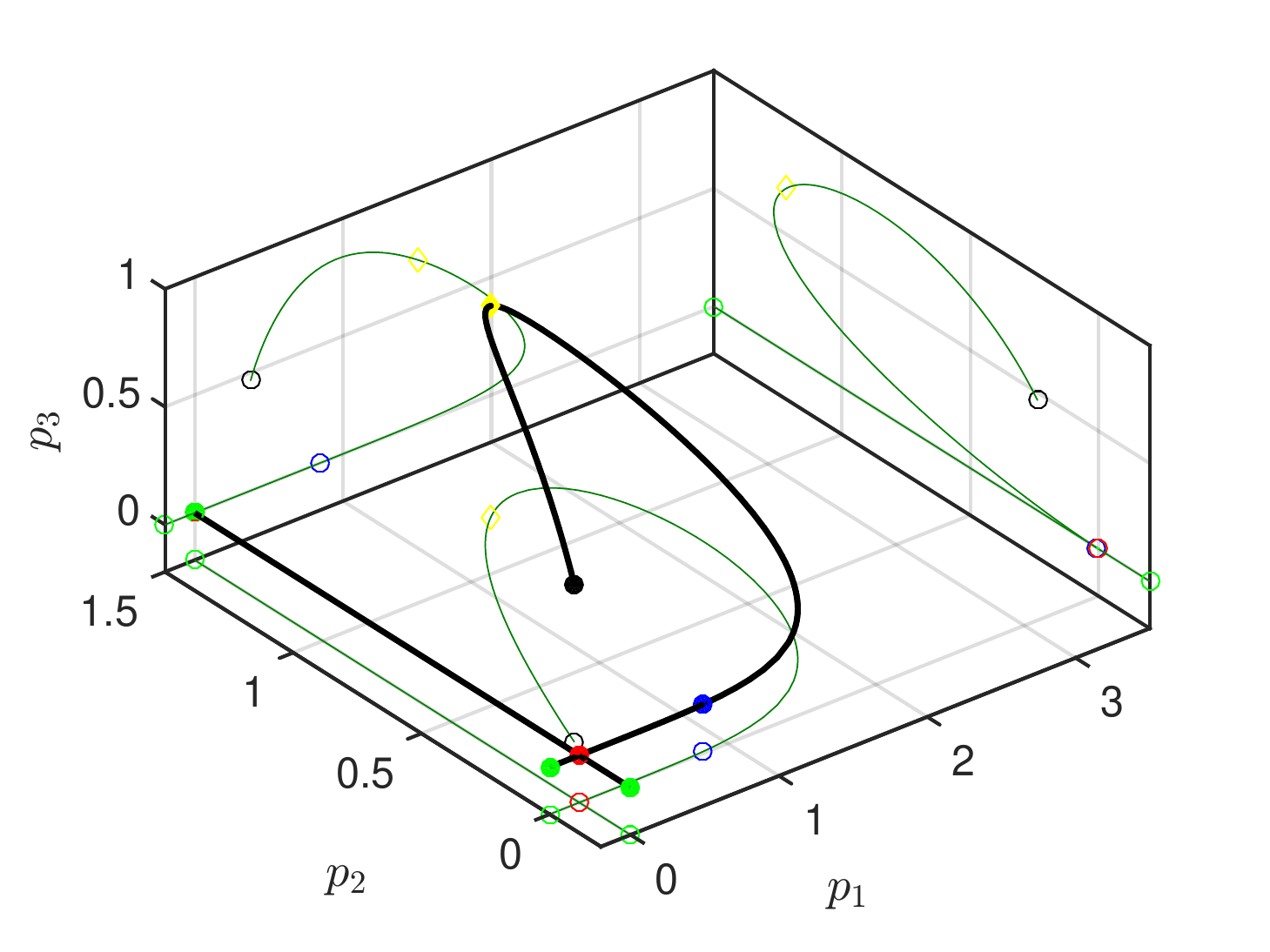}
\includegraphics[width=3.0in]{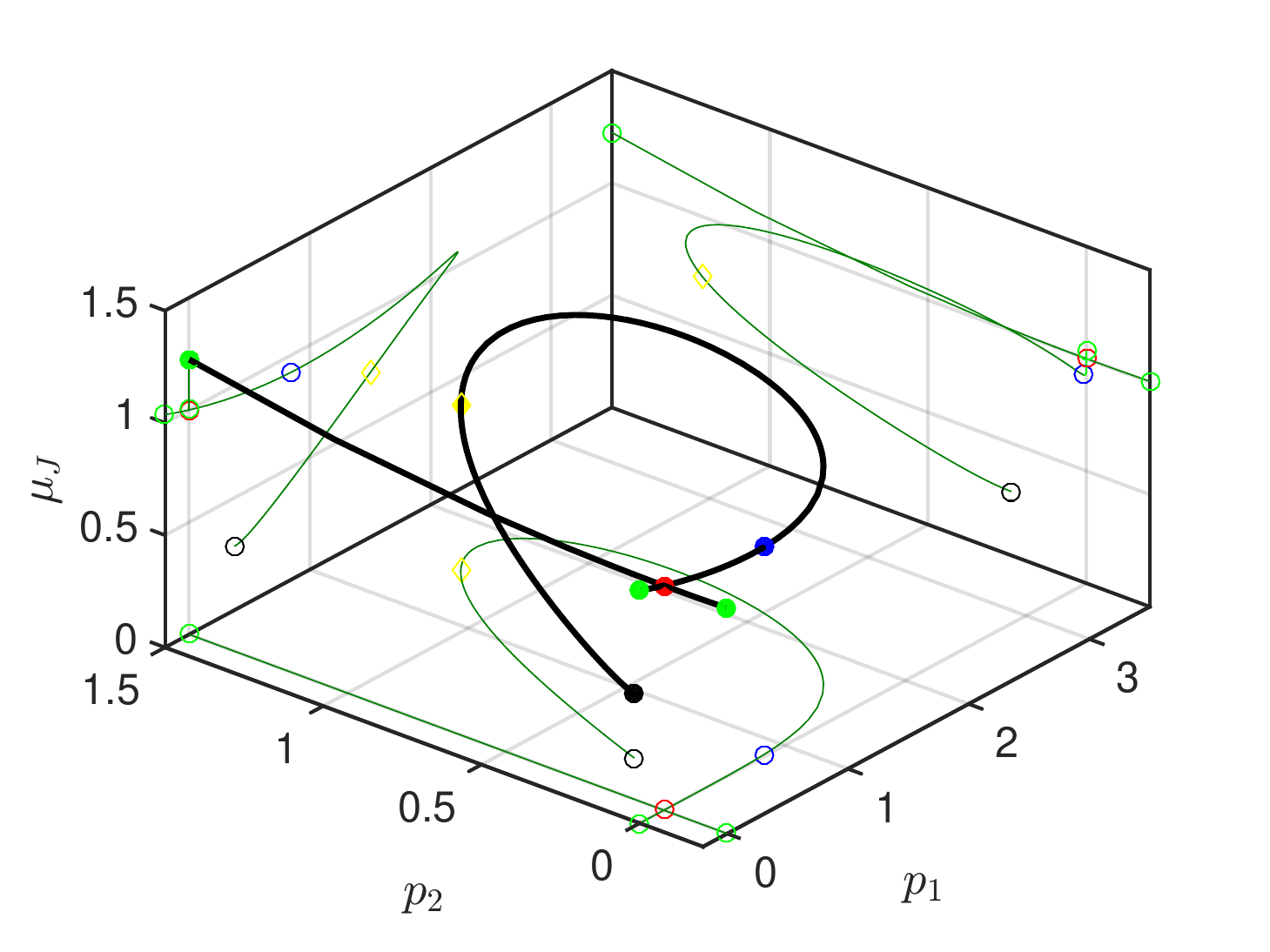}\\
\includegraphics[width=3.0in]{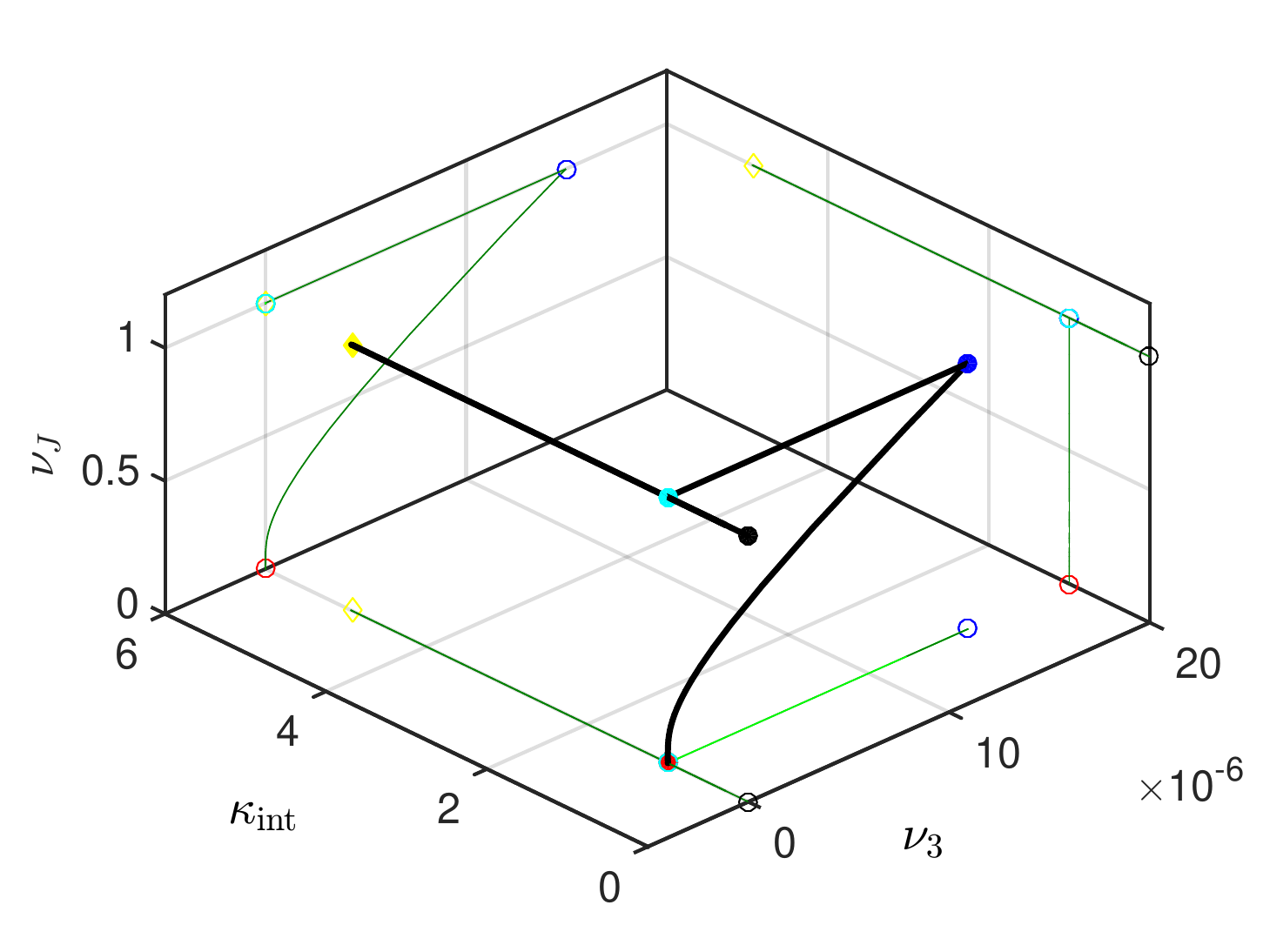}
\includegraphics[width=3.0in]{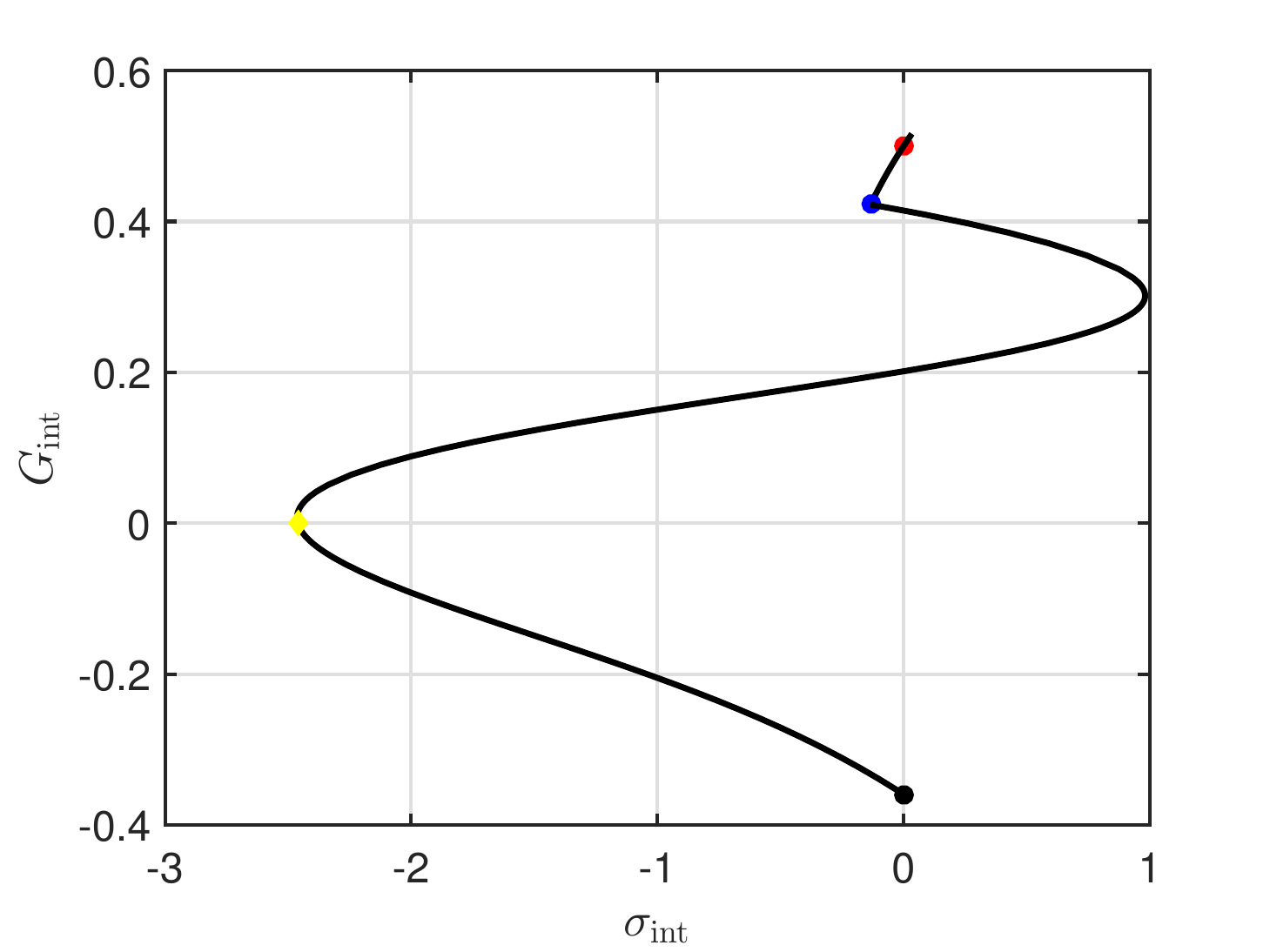}
\caption{Projections of continuation paths  associated with a {successive} search for stationary solutions. {Here, dark green thin lines and hollow markers are used to denote projections of black thick lines and filled markers in three-dimensional space onto the three coordinate planes.} Starting at $u_0=(0,0,0.1,0)$ and holding $\mu_{p_3}$, $\nu_{p_1}$, $\nu_{p_2}$, and $\kappa_\mathrm{int}$ fixed at $0$, $0$, $0$, and $1$, respectively, a fold point in $\mu_J$, denoted by red dots, is detected along the first solution manifold in the subspace with vanishing Lagrange multipliers. Along the secondary manifold, blue dots denote locations where $\nu_{J}=1$. With $\nu_J$, $\nu_{p_1}$, $\nu_{p_2}$, and $\kappa_\mathrm{int}$ fixed at $1$, $0$, $0$, and $1$, respectively, the cyan dots denote locations where $\nu_{p_3}=0$. As seen in the bottom left panel, $\nu_{p_3}$ decreases monotonically from $1.5\times 10^{-5}$ to $0$. Finally, the terminal points (black dots) on the fourth manifold denote the stationary points where $\kappa_{\mathrm{int}}=0$ assuming fixed $\nu_J$, $\nu_{p_1}$, $\nu_{p_2}$, and $\nu_{p_3}$. Notably, the last manifold crosses $G_{\mathrm{int}}=0$ at a regular point (yellow diamond) with $\sigma_{\mathrm{int}}\neq0$. The green dots denote points on the boundary of the computational domain.}
\label{fig:cont-path-infeasible}
\end{figure}

In the second case, we perform continuation of solutions to the original boundary-value problem from $x(t)=0$, $p_1=0$, $p_2=0.1$, and $p_3=0$ under variations in $p_1$. {As seen in~\cref{fig:cont-path-feasible}, three fold points are detected and we locate two local extrema (${FP}_{2,3}$) in the value of $\mu_J$ within the feasible region.}
We let $u_0$ equal the ${FP}_{3}$ of this initial run. Next, we consider continuation along the one-dimensional solution manifold with trivial Lagrange multipliers obtained by fixing $\mu_{p_2}$, $\mu_{p_3}$, $\nu_{p_1}=0$, and $\kappa_\mathrm{int}=0$ and allowing the remaining continuation parameters to vary. Branch switching from ${FP}_{2}$ or ${FP}_{3}$ should allow for continuation along secondary branches of solutions with nonzero Lagrange multipliers and unchanged values $u$. This is verified by the numerical results in~\cref{fig:cont-path-feasible}. 

Starting from ${FP}_3$, we drive $\nu_J$ to $1$. {As predicted by \cref{cor2}, $u$ is preserved in this run because $\mathbb{P}=\emptyset$.} Next, we fix $\nu_J=1$ and allow  $\mu_{p_2}$ to vary until $\nu_{p_2}=0$. In the final stage, we fix $\nu_{p_2}=0$ and allow $\mu_{p_3}$ to vary until $\nu_{p_3}=0$. In each of these runs the terminal values are approached monotonically. The final point corresponds to the local extremum at $p_1=0.37722$, $p_2=0.23782$, $p_3=0.46761$, from which we obtain $J=0.28459$. Starting from ${FP}_2$, we again drive $\nu_J$ to $1$ monotonically. However, once we fix $\nu_J=1$ and allow $\mu_{p_2}$ to vary, we observe a failure of the Newton solver to converge as we approach a singular point on $G_\mathrm{int}=0$ before $\nu_{p_2}=0$. This point is denoted by the magenta squares in~\cref{fig:cont-path-feasible}. {If we allow larger computational domain, it turns out that we can arrive at $\nu_{p_2}=0$ by continuting in the other direction.}

\begin{figure}[H]
\centering
\includegraphics[width=3.0in]{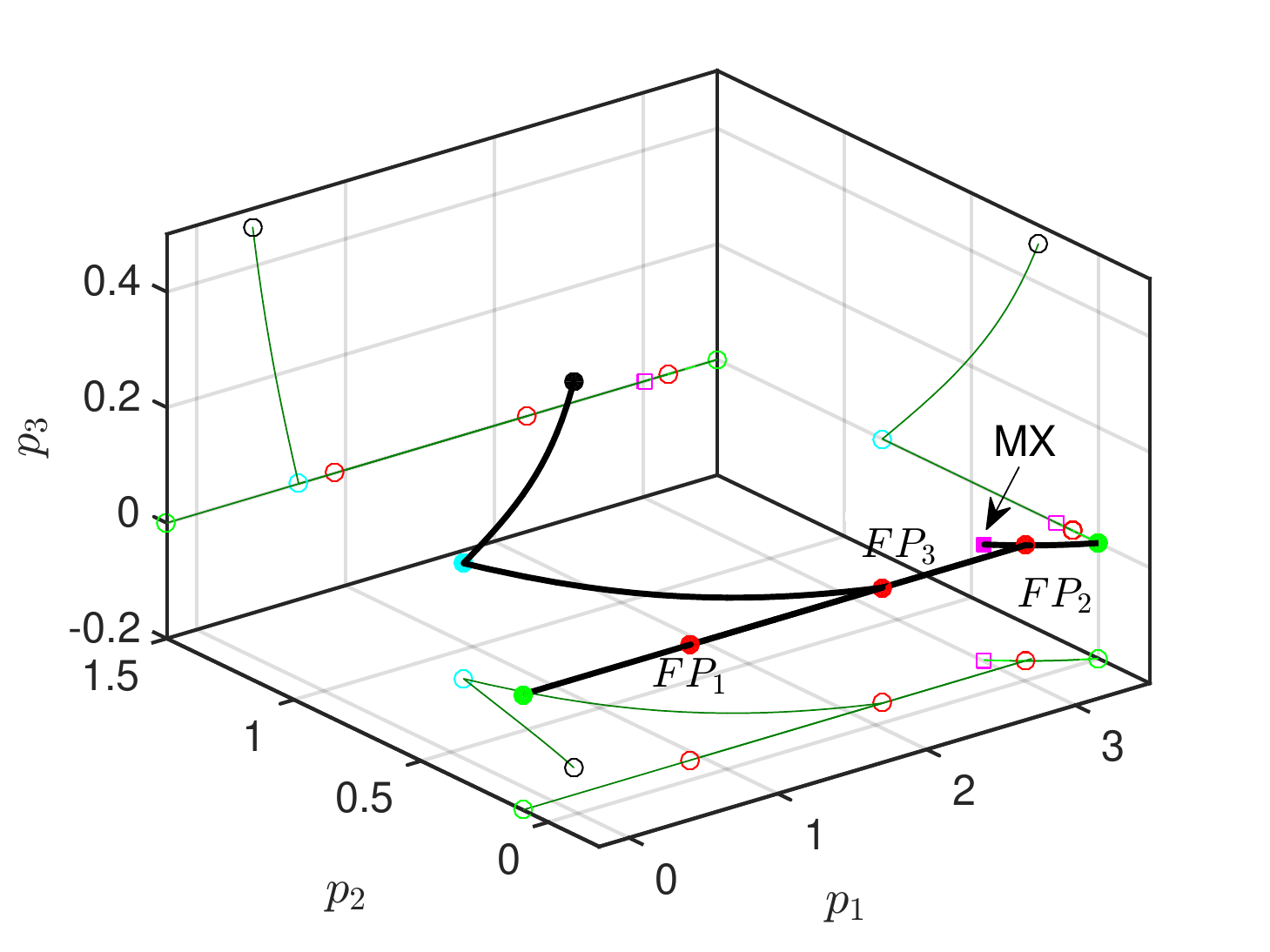}
\includegraphics[width=3.0in]{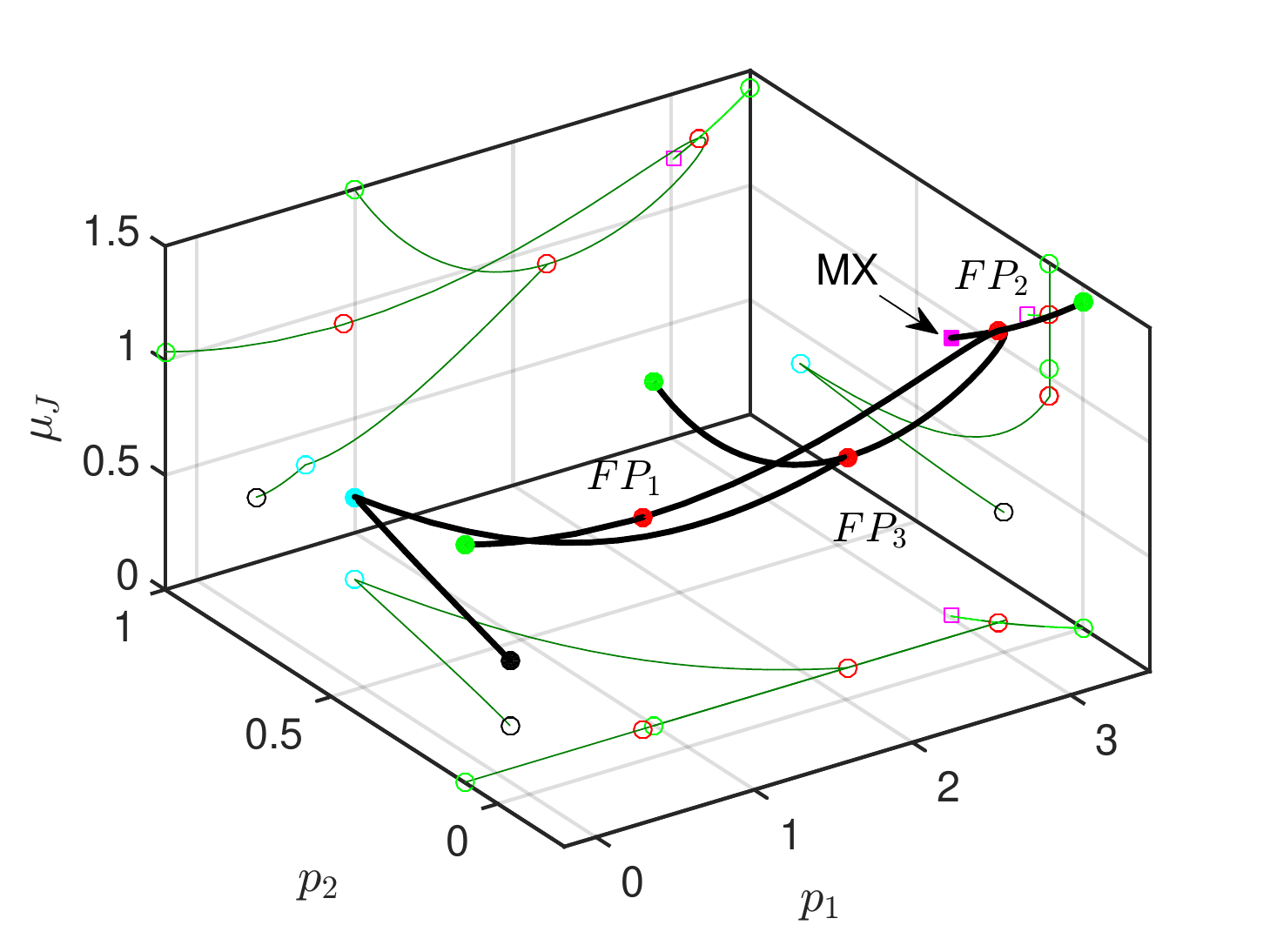}\\
\includegraphics[width=3.0in]{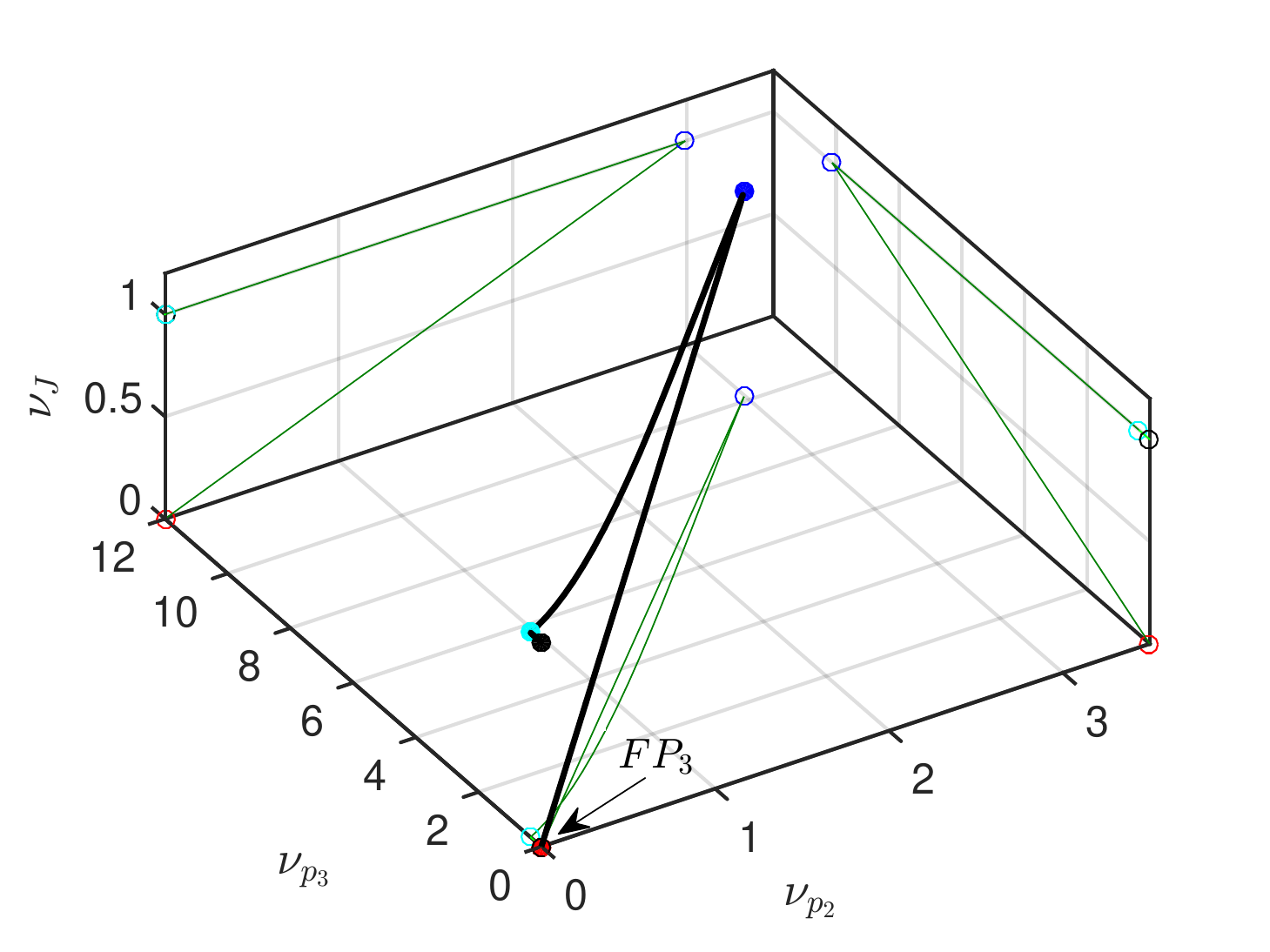}
\includegraphics[width=3.0in]{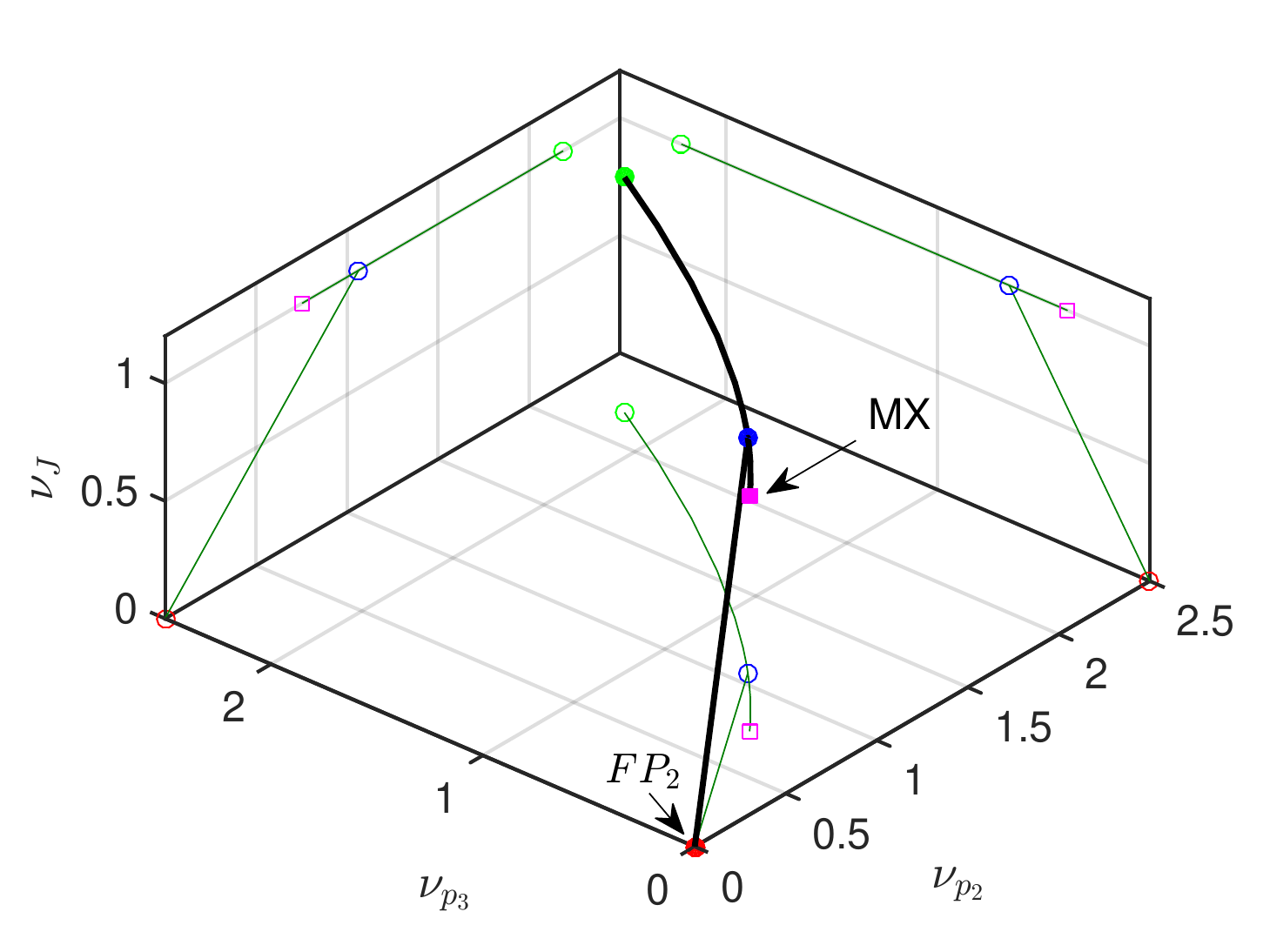}
\caption{Projections of continuation paths  associated with a {successive} search for stationary solutions. {Here, dark green thin lines and hollow markers are used to denote projections of black thick lines and filled markers in three-dimensional space onto the three coordinate planes.} A preliminary run is conducted to obtain an initial solution in the feasible region.
Starting at $\tilde{u}_0=(0,0,0.1,0)$, three fold points (denoted by  ${FP}_1$,  ${FP}_2$ and  ${FP}_3$ and identified by red dots) in $\mu_J$ are detected during continuation of the original boundary value problem (ignoring the inequality constraint and adjoint conditions) with $\mu_{p_2}$ and $\mu_{p_3}$ fixed at $0.1$ and $0$, respectively. The last two fold points are located within the feasible region. Taking $FP_3$ as $u_0$, both $FP_2$ and $FP_3$ correspond to branch points during continuation with $\mu_{p_2}$, $\mu_{p_3}$, $\nu_{p_1}$, and $\kappa_\mathrm{int}$ fixed at $0.1$, $0$, $0$, and $0$, respectively, along the first solution manifold in the subspace of vanishing Lagrange multipliers. Along each of the secondary manifolds emanating from ${FP}_2$ (bottom right) and ${FP}_3$ (bottom left), respectively, blue dots denote locations where $\nu_{J}=1$. Continuation along the tertiary manifolds with $\nu_J$, $\mu_{p_3}$, $\nu_{p_1}$, and $\kappa_\mathrm{int}$ fixed at $1$, $0$, $0$, and $0$, respectively, reaches a point with $\nu_{p_2}=0$ (cyan dot) in the bottom left panel, but terminates at a singular point (magenta squares labeled by MX) or at a point on the boundary of the computational domain before reaching $\nu_{p_2}$ in the bottom right panel. In the case of the bottom left panel, the terminal points (black dots) on the manifolds obtained from continuation with $\nu_J$, $\nu_{p_1}$, $\nu_{p_2}$, and $\kappa_\mathrm{int}$ fixed at $1$, $0$, $0$, and $0$, respectively, denote stationary points where $\nu_{p_3}=0$. The green dots denote points on the boundary of the computational domain.}
\label{fig:cont-path-feasible}
\end{figure}

\subsection{Optimal control}
Consider the problem of minimizing the objective functional
\begin{equation}
\label{eq:obj}
    J = \int_0^{2} \left(100\theta^2(t)+40\dot{\theta}^2(t)+u^2(t)\right) \mathrm{d}t
\end{equation}
on solutions of the nonlinear inverted pendulum dynamical system~\cite{invert_pend} shown in \cref{fig:invert_pend_model} and governed by the differential equations in terms of displacement $x$, rotation $\theta$ and input $u$
\begin{gather}
\label{eq:invert_pend_ode}
    (M+m)\ddot{x}-ml\dot{\theta}^2\sin\theta+ml\ddot{\theta}\cos\theta=u,\quad
    m\ddot{x}\cos\theta+ml\ddot{\theta}-mg\sin\theta=0,
\end{gather}
and initial conditions $\theta(0)=0.1$, $\dot{\theta}(0)=x(0)=\dot{x}(0)=0$, subject to integral inequality constraints of the form
\begin{equation}
\label{eq:input}
    \|u\|^2:= \int_0^{t_f} u^2(t) dt\leq E_c
\end{equation}
or
\begin{equation}
\label{eq:output}
    \|y\|^2:=\int_0^{t_f} (\theta^2(t)+x^2(t)) dt\leq Y_c,
\end{equation}
where $E_c$ and $Y_c$ are input and output thresholds, respectively. Following the scheme in~\cite{staged_adjoint}, we accommodate this optimal control problem within the proposed optimization framework by parameterizing the control input using a 10-term truncated Chebyshev-polynomial expansion and let the problem parameters $p_1,\ldots,p_{10}$ denote the unknown coefficients of the expansion. In the notation of previous sections, it follows that $d=10$ and $l=11$. We restrict attention throughout to variations in the elements of $\nu_p$ that approach $0$ monotonically. In all the numerical results reported here, $M=2$, $m=0.1$, $l=0.5$, and $g=9.81$.
\begin{figure}[ht!]
\centering
\includegraphics[width=0.45\textwidth]{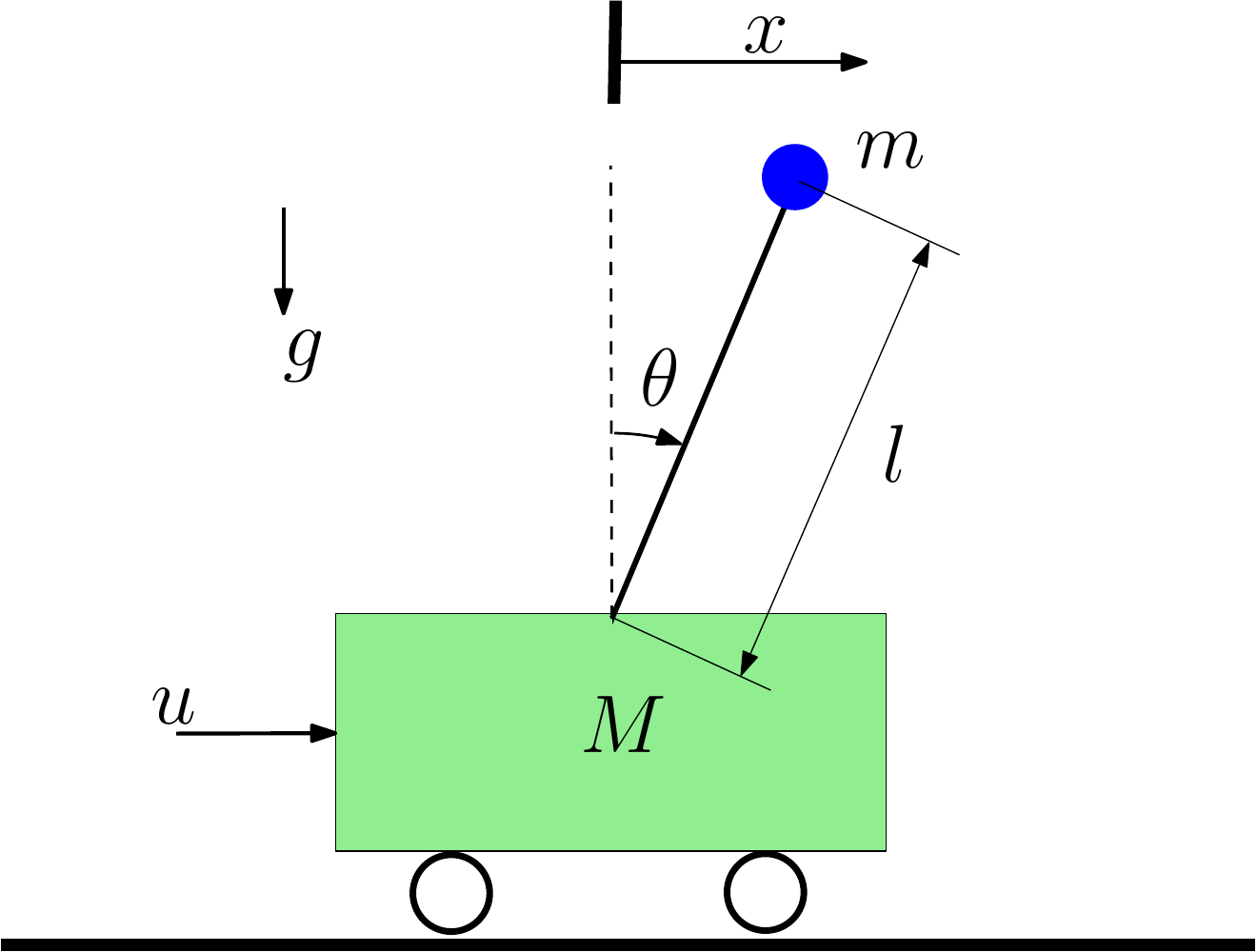}
\caption{Schematic of dynamical system corresponding to \eqref{eq:invert_pend_ode} (adapted from~\cite{invert_pend}).}
\label{fig:invert_pend_model}
\end{figure}

% discuss three different cases: optimal control and no control---  first see the optimization effects
% 
In order to explore the effects of boundedness of input and output, we first solve the optimal control problem in the absence of such bounds. It follows that $q=0$ and $\mathbb{P}=\emptyset$. We set $p_{1,0}=\cdots=p_{10,0}=0$ and construct the initial solution $x_0(t)$ via forward simulation. For this optimization problem, we follow the successive continuation approach by
\begin{enumerate}
\item taking $\mathbb{I}=\{3,...,11\}$, i.e., fixing $\mu_{p_{\{2,...,10\}}}$ and allowing $\mu_{p_1}$ vary to yield a one-dimensional manifold. A fold point of $\mu_J$ is detected along this manifold;
\item branching off from the fold point and driving $\nu_J$ to $1$; 
\item fixing $\nu_J=1$, and then successively allowing each of the remaining elements of $\mu_p$ to vary (in order of the expansion of $u(t)$), and fixing the corresponding element of $\nu_p$ once it equals $0$.
\end{enumerate}
The resulting optimal trajectories and control input are represented by solid lines in~\cref{fig:opt-traj} and~\cref{fig:opt-input}, respectively. For this optimal solution, the input integral $\|u\|^2=3.9457$, the output integral $\|y\|^2=3.7115\times 10^{-2}$, and $J=5.5759$.

\begin{figure}[h]
\centering
\includegraphics[width=3.0in]{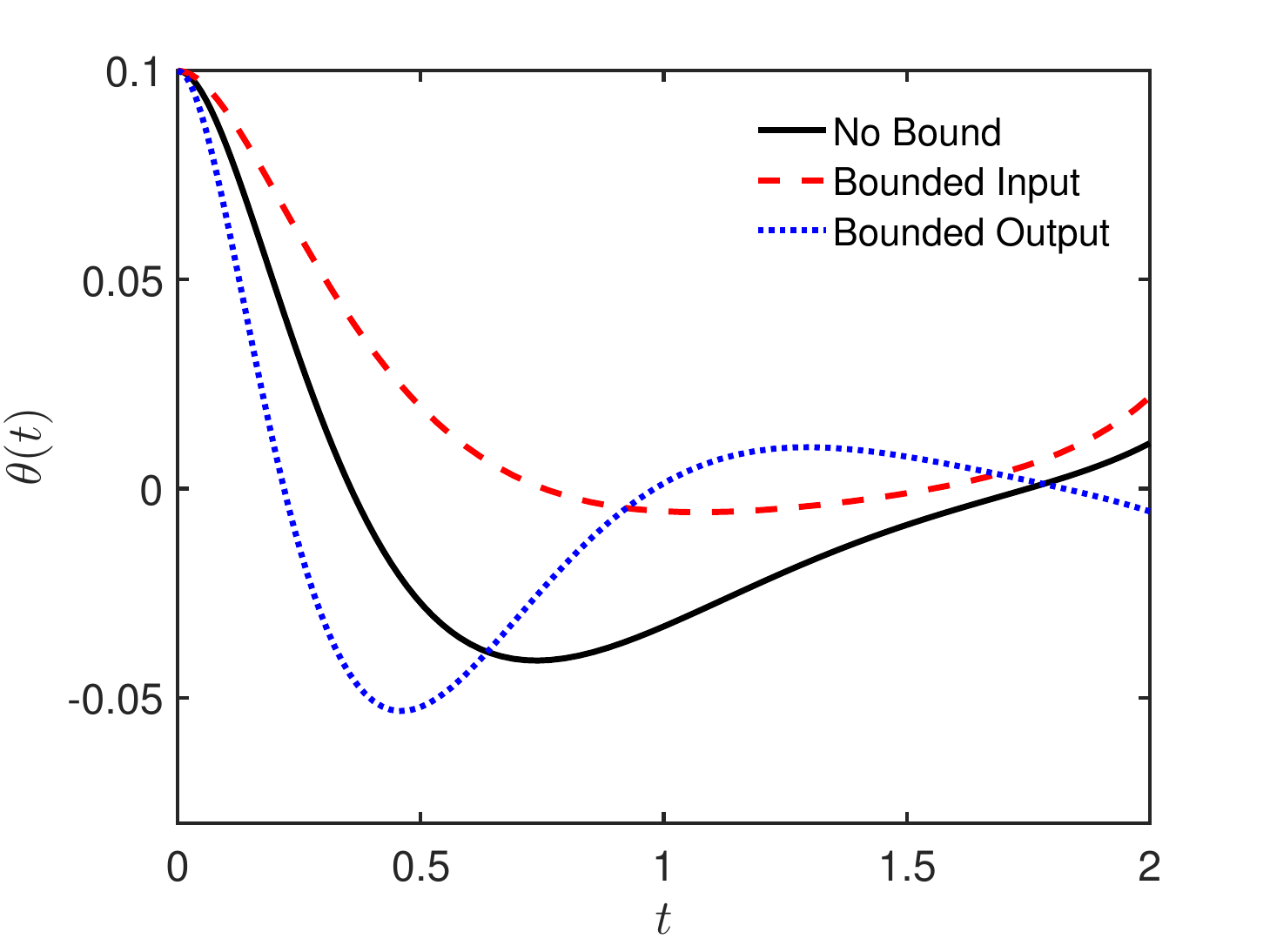}
\includegraphics[width=3.0in]{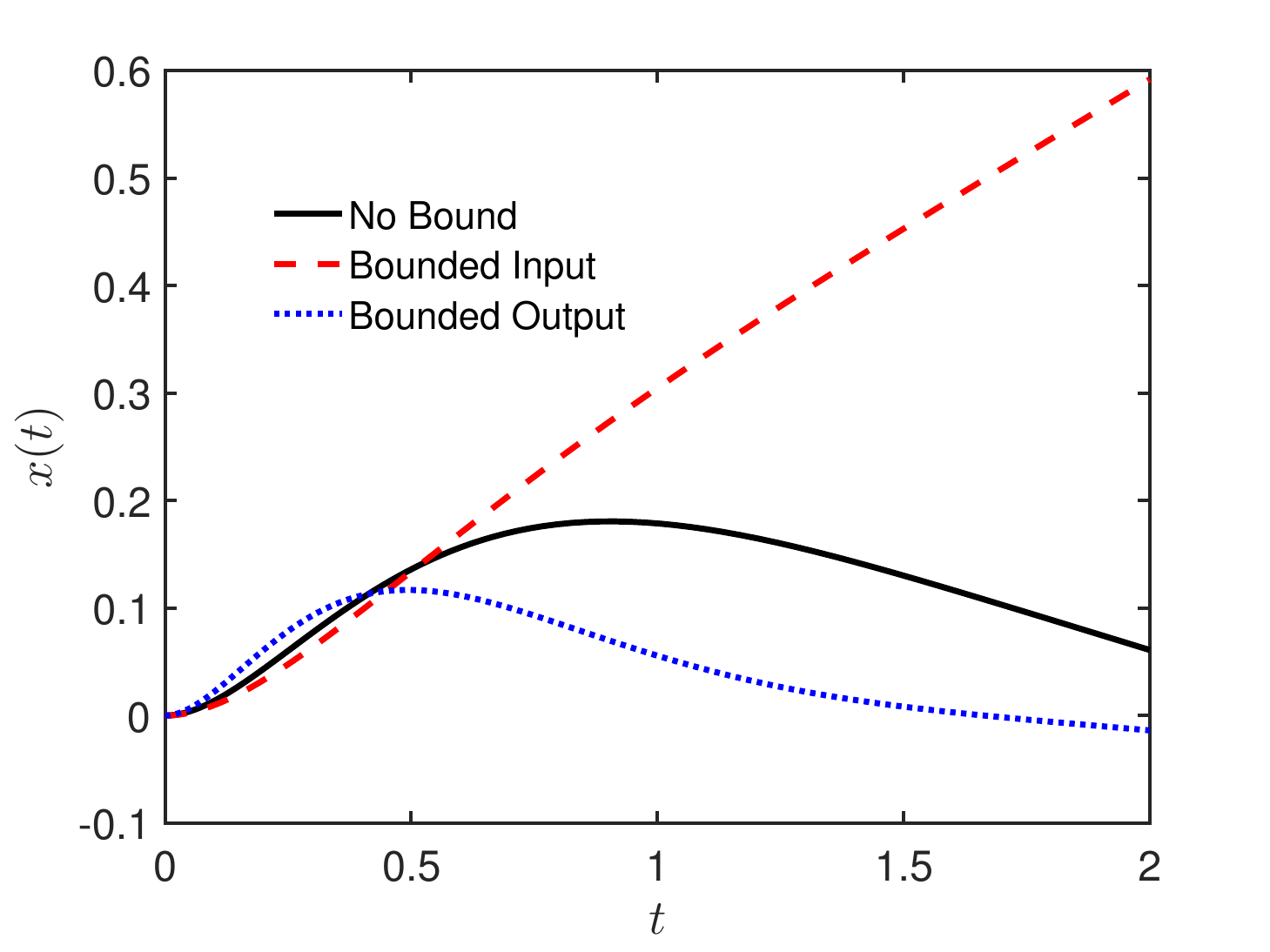}
\caption{Optimal time histories for $\theta(t)$ (left panel) and $x(t)$ (right panel) in the case without inequality constraints (solid lines), with input integral inequality (dashed lines), and with output integral inequality (dotted lines).}
\label{fig:opt-traj}
\end{figure}

\begin{figure}[ht!]
\centering
\includegraphics[width=0.6\textwidth]{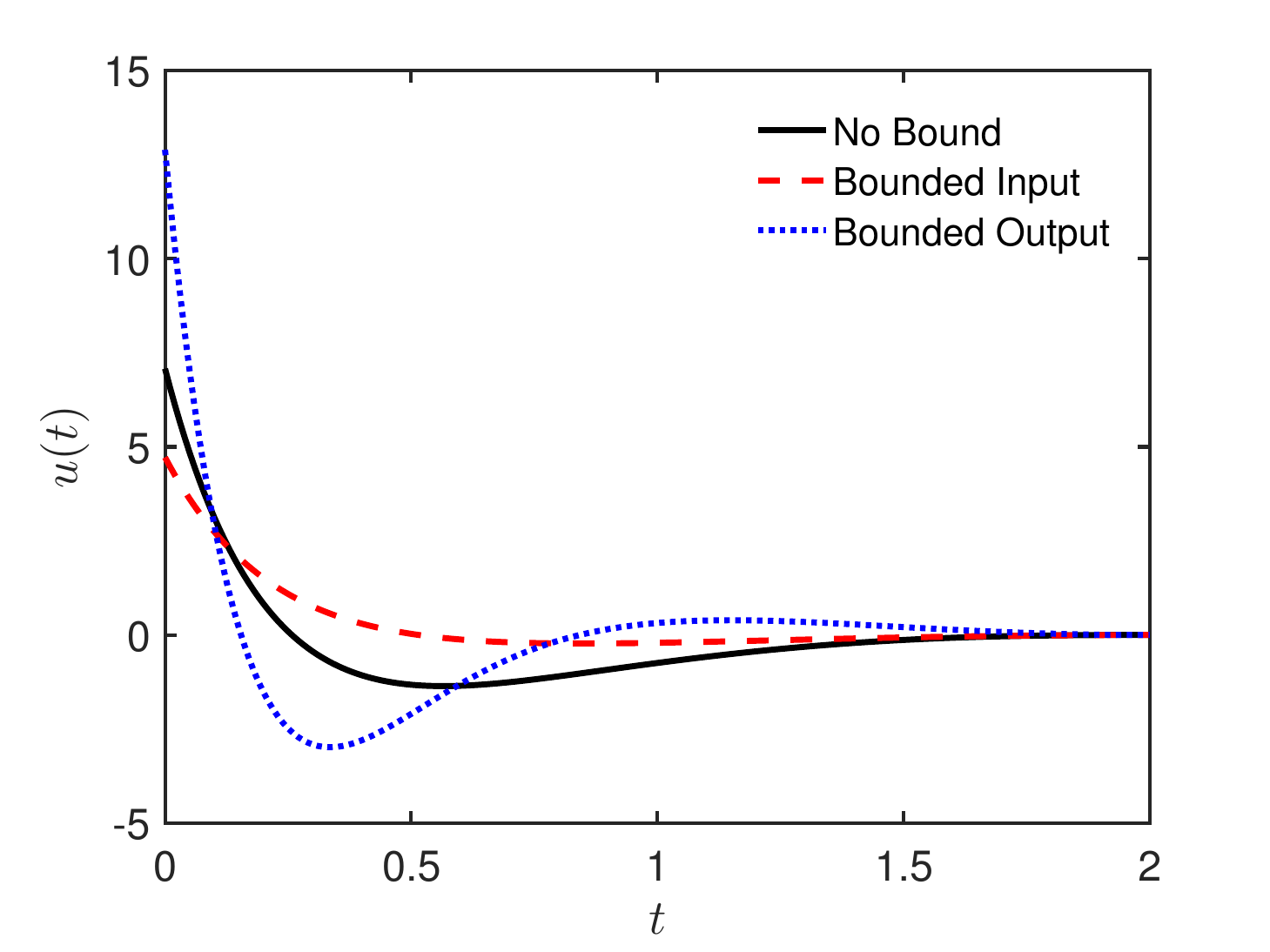}
\caption{Optimal time histories $u(t)$ for the control input in the case without inequality constraints (solid lines), with input integral inequality (dashed lines), and with output integral inequality (dotted lines).}
\label{fig:opt-input}
\end{figure}

We next consider optimization subject to the integral inequality~\eqref{eq:input} with $E_c=2$ to explore the effects of the boundedness of control input. Note that the integral inequality is formulated as an algebraic inequality in our framework because of the control parameterization. It follows that $q=1$. Let $\kappa_{\mathrm{input}}$ denotes the continuation parameter for the NCP condition of~\eqref{eq:input}. With initial parameters $p_1=p_2=3$, $p_3=p_4=1$ and $p_{\{5,\ldots,10\}}=0$, we construct an initial solution $x(t)$ located in te infeasible region by forward simulation. Since $\mathbb{P}=\{1\}$, the inactive problem parameter set $\mathbb{I}$ should be selected in such a way that $|\mathbb{I}|=d-1-|\mathbb{P}|=8$. To this end, we apply the successive continuation approach by
\begin{enumerate}
\item taking $\mathbb{I}=\{4,...,11\}$, i.e., fixing $\mu_{p_{\{3,...,10\}}}$ and $\kappa_\mathrm{input}$, and allowing $\mu_{p_1}$ and $\mu_{p_2}$ to vary to yield a one-dimensional manifold. Several fold points of $\mu_J$ are detected along this manifold;
\item branching off from the fold associated with the smallest value of $\mu_J$, and driving $\nu_J$ to $1$; 
\item fixing $\nu_J=1$, and then successively allowing each of the remaining elements of $\mu_p$ to vary (in order of the expansion of $u(t)$), and fixing the corresponding element of $\nu_p$ once it equals $0$;
\item driving $\kappa_\mathrm{input}$ to $0$.
\end{enumerate}
The resulting optimal trajectories and control input are represented by dashed lines in~\cref{fig:opt-traj} and~\cref{fig:opt-input}, respectively. With bounded control input, we observe that $x(t)$ oscillates with larger amplitude (see the right panel of~\cref{fig:opt-traj}) and the objective functional is increased to $J=11.774$.

We finally explore the effects of the boundedness of output by considering optimization subject to the integral inequality~\eqref{eq:output} with $Y_c=0.01$. With initial parameters $p=0$, we use forward simulation to construct an initial solution $x(t)$ that is again located in the infeasible region. Following the same approach as in the case of bounded input, we obtain the optimal trajectories and control input represented by dotted lines in~\cref{fig:opt-traj} and~\cref{fig:opt-input}, respectively. The bound on the output dampens the vibration of $x(t)$, as can be seen in the right panel of~\cref{fig:opt-traj}. For this optimal solution, we have $\|u(t)\|^2=8.8954$, indicating that more control input is required to ensure that the output stays within the given bound. As a consequence, the objective functional is increased to $J=9.4105$.

\section{Conclusions}
\label{sec:conclusions}
As advertised in the introduction, this paper has developed a rigorous framework within which the successive continuation paradigm for single-objective-function constrained optimization of Kern{\'e}vez and Doedel~\cite{DKK91} may be extended to the case of simultaneous equality and inequality constraints. The discussion has also shown that the structure of the corresponding continuation problems fits naturally with the \textsc{coco} construction paradigm. Indeed, a forthcoming release of \textsc{coco} will include documented support for this extended functionality. The finite- and infinite-dimensional examples illustrate the general methodology as well as number of opportunities for further development and automation.

In this study, the nonsmooth Fischer-Burmeister function was used to express complementarity constraints in the form of equalities. Notably, the singularity of this function at the origin was associated with singular points along various solution manifolds and a potential failure of the Newton solver to converge. Some generalized Newton methods have been developed to tackle such a singularity. One approach is to replace regular derivatives by Clarke subdifferentials~\cite{comp} or other generalized Jacobians~\cite{nonsmooth_newton}. An alternative approach is to approximate the nonsmooth problem by a family of smooth problems~\cite{comp,ncp}. More specifically, the NCP function $\chi$ may be approximated by a family of smooth approximants $\chi_{\epsilon}$, parameterized by the scalar $\epsilon$, such that the solutions to the perturbed problems $\chi_{\epsilon}=0$ form a smooth trajectory parameterized by $\epsilon$ that converges to the solution of $\chi=0$ as $\epsilon\to0$. Such a smoothing approach is analogous to the homotopy approach used in this paper to satisfy inequality constraints. The pesky singularity encountered in this study could thus be avoided by further expanding the successive continuation technique to a stage during which $\epsilon$ is driven to $0$.

It has been tacitly assumed that each successive stage of continuation is able to drive the appropriate continuation parameters to their desired values, preferably monotonically. But the examples showed that this may not be possible within a given computational domain, or may only be possible by occasionally bearing in a direction away from the desired values. Similar observations were already made in~\cite{staged_adjoint} and generalize in this paper to the relaxation parameters $\kappa$. Indeed, in the presence of inequality constraints, we observed instances in which solutions branches would terminate on singular points on zero-level surfaces of the inequality functions before the desired end points were reached. In some cases, the pseudo-arclength continuation algorithm automatically switched to a separate branch in such a zero-level surface allowing us to continue to drive a component of $\nu$ or $\kappa$ to its desired value. In other cases, this did not happen automatically, and we did not attempt to locate the secondary branch manually.

On a related note, we observe that in each successive stage of continuation implemented in the examples, only one component of $\nu$ or $\kappa$ at a time was driven to its desired value. Alternatively, one might imagine first driving the component of $\nu$ associated with the objective function to $1$ and then attempting to drive multiple remaining components of $\nu$ and $\kappa$ to $0$ simultaneously. Furthermore, although this did not happen in our examples, one imagines the possibility that continuation in a zero-level surface of an inequality function would terminate at a singular point (with the corresponding component of $\sigma$ equal to $0$). Further continuation might then need to branch off the zero-level surface in order to locate the desired stationary point. Clearly, any automated search for stationary points would need to consider these many possibilities.

Finally, we note that the approach in this paper is restricted to finite-dimensional inequality constraints and does not automatically generalize to the infinite-dimensional case. If the latter is discretized before the formulation of adjoints~\cite{flow,lauss-discrete}, then the present framework is again applicable. If, as advocated here and in~\cite{staged_adjoint}, the formulation of adjoints precedes discretization, an appropriate and consistent discretization scheme for the original equations, adjoint equations, and complementarity constraints needs to be carefully established~\cite{opt-control,nonsmooth_newton}. In either case, a high-dimensional vector $\kappa$ of relaxation parameters may result. Driving each component of $\kappa$ to zero one-by-one could be very time-consuming, so a smarter search strategy (as alluded to in the previous paragraph) would be desirable.

% give a brief summary, and then move to some future directions, including the singularity issues, different ncp functions, different iteration methods, path constraints

\appendix
\section{Essential lemmas}
\label{sec:essential lemmas}
Let $\mathcal{R}(L)$ and $\mathcal{N}(L)$ denote the range and nullspace of a linear map $L:V\mapsto W$. We say that $L$ is full rank if $V=V_1\oplus\mathcal{N}(L)$ and $L\big|_{V_1}$ is a bijection onto $W$. This holds, for example, if $\mathcal{R}(L)=W$ and $\mathrm{dim}\left(\mathcal{N}(L)\right)<\infty$. By the implicit function theorem, it follows that if $F(u_0)=0$ for a continuously Frechet-differentiable map $F$ between two Banach spaces $V$ and $W$, and if $DF(u_0)$ is full rank with finite-dimensional nullspace, then the roots of $F$ near $u_0$ lie on a manifold with tangent space spanned by $\mathcal{N}\left(DF(u_0)\right)$.

Recall the construction of the augmented function $F_\mathrm{aug}$ in \eqref{eq:aug}, the restriction $F_\mathrm{rest}$ obtained by fixing the values of $\mu_{\mathbb{I}}$, $\nu_{\mathbb{J}}$, and $\kappa$ to $\Psi_{\mathbb{I}}(u_0)$, $0$, and $\kappa_0$, respectively, and the reduced function $F_\mathrm{red}$ in \eqref{eq:red} obtained by retaining only entries corresponding to constraints on $u$ in $F_\mathrm{rest}=0$. Assume throughout that $F_\mathrm{red}(u_0)=0$ and, unless otherwise noted, that $\{i:G_i(u_0)=0\}=\emptyset$.

\begin{lemma}\label{lem1}
Suppose that the linear map $DF_\mathrm{red}(u_0)$ is full rank with one-dimensional nullspace. This also holds for $DF_\mathrm{rest}(u_0,0,0,0,\Psi_1(u_0),\Psi_{\mathbb{J}}(u_0),0,0)$ provided that $\left(D\Psi_1(u_0)\right)^\ast$ is linearly independent of $\left(DF_\mathrm{red}(u_0)\right)^*$. 
\end{lemma}
\begin{proof}
By assumption, the inverse image of every $w\in Y\times\mathbb{R}^{|\mathbb{I}|}\times\mathbb{R}^{|\mathbb{P}|}$ is a one-dimensional affine subspace of $U$ of the form $v\oplus \mathcal{N}\left(DF_\mathrm{red}(u_0)\right)$ for some $v$. By the theory of Fredholm operators, it follows that $\mathcal{N}\left(\left(DF_\mathrm{red}(u_0)\right)^\ast\right)=0$ and there exists a one-dimensional subspace $\Sigma$ of $U^\ast$, such that $U^\ast=\Sigma\oplus\mathcal{R}\left(\left(DF_\mathrm{red}(u_0)\right)^\ast\right)$. If $\left(D\Psi_1(u_0)\right)^\ast$ is linearly independent of $\left(DF_\mathrm{red}(u_0)\right)^*$, it follows that $U^\ast=\mathcal{R}\left(\left(D\Psi_1(u_0)\right)^\ast\right)\oplus\mathcal{R}\left(\left(DF_\mathrm{red}(u_0)\right)^\ast\right)$. The claim then follows by inspection of the image of the linear map $DF_\mathrm{rest}(u_0,0,0,0,\Psi_1(u_0),\Psi_{\mathbb{J}}(u_0),0,0)$.
\end{proof}
\begin{corollary}\label{cor1}
Suppose that the linear map $DF_\mathrm{red}(u_0)$ is full rank with one-dimensional nullspace and that $\left(D\Psi_1(u_0)\right)^\ast$ is linearly independent of $\left(DF_\mathrm{red}(u_0)\right)^*$. It follows that the roots of $F_\mathrm{rest}$ sufficiently close to $(u_0,0,0,0,\Psi_1(u_0),\Psi_{\mathbb{J}}(u_0),0,0)$ lie on a one-dimensional manifold of points of the form $(u,0,0,0,\Psi_1(u),\Psi_{\mathbb{J}}(u),0,0)$ for some root $u$ of $F_\mathrm{red}$.
\end{corollary}
\begin{lemma}\label{lem2}
Suppose that the linear map $DF_\mathrm{red}(u_0)$ is full rank with one-dimensional nullspace and that $u_0$ is a stationary point of $\Psi_1$ along the corresponding one-dimensional solution manifold. Then, generically, $(u_0,0,0,0,\Psi_1(u_0),\Psi_{\mathbb{J}}(u_0),0,0)$ is a branch point of $F_\mathrm{rest}$ through which runs a secondary one-dimensional solution manifold, locally parameterized by $\eta_1$, along which $\lambda$, $\eta_\mathbb{I}$, and $\sigma_\mathbb{P}$ vary.
\end{lemma}
\begin{proof}
By assumption, there exists a unique vector $z\in \left(Y\times\mathbb{R}^{|\mathbb{I}|}\times\mathbb{R}^{|\mathbb{P}|}\right)^\ast$ such that $\left(D\Psi_1(u_0)\right)^\ast=\left(DF_\mathrm{red}(u_0)\right)^*z$. It follows by inspection of its image that the linear map $DF_\mathrm{rest}(u_0,0,0,0,\Psi_1(u_0),\Psi_{\mathbb{J}}(u_0),0,0)$ has a two-dimensional nullspace and is no longer full rank. Indeed, for $u\approx u_0$, roots of $F_\mathrm{rest}$ correspond to solutions to the system of equations
\begin{equation}\label{lemeq}
\left\{\begin{array}{c}\Phi(u)=0,\\\Psi_\mathbb{I}(u)-\Psi_\mathbb{I}(u_0)=0,\\K_\mathbb{P}\left(\sigma,-G(u)\right)-\kappa_{0,\mathbb{P}}=0,\\\left(D\Phi(u)\right)^\ast\lambda+\left(D\Psi_1(u)\right)^\ast\eta_1+\left(D\Psi_\mathbb{I}(u)\right)^\ast\eta_\mathbb{I}+\left(DG_\mathbb{P}(u)\right)^\ast\sigma_\mathbb{P}=0.
\end{array}\right.
\end{equation}
For every $\sigma_\mathbb{P}$ with $\|\sigma_\mathbb{P}\|=\epsilon\ll 1$, there exists a one-dimensional manifold of solutions to the first three equations. By continuity, each such manifold generically contains a unique stationary point of $\Psi_1$ close to $u_0$. At each such point, the fourth equation may be solved for $\lambda$, $\eta_\mathbb{I}$, and $\sigma_\mathbb{P}$ for given $\eta_1>0$ such that the vector $(-\lambda/\eta_1,-\eta_\mathbb{I}/\eta_1,-\sigma_\mathbb{P}/\eta_1)$ is close to $z$. It follows that there exists a least one such point where the two values of the vector $\sigma_\mathbb{P}$ agree for some $0<\eta_1\ll 1$. The claim follows by considering variations in $\epsilon$.
\end{proof}
\begin{corollary}\label{cor2}
Under the assumptions of Lemma~\ref{lem2}, $u$ varies along the secondary branch only if $\mathbb{P}\neq\emptyset$.
\end{corollary}
\begin{lemma}\label{lem3}
Suppose that the linear map $DF_\mathrm{red}(u_0)$ is a bijection, i.e., that $u_0$ is a locally unique root of $F_\mathrm{red}$. Then, the point $(u_0,0,0,0,\Psi_1(u_0),\Psi_{\mathbb{J}}(u_0),0,0)$ lies on a one-dimensional manifold of solutions to $F_\mathrm{rest}=0$ locally parameterized by $\eta_1$, along which $\lambda$, $\eta_\mathbb{I}$, and $\sigma_\mathbb{P}$ vary.
\end{lemma}
\begin{proof}
By assumption, there exists a unique inverse image $v\in U$ for every $w\in Y\times\mathbb{R}^{|\mathbb{I}|}\times\mathbb{R}^{|\mathbb{P}|}$. By the standard theory of Fredholm operators, $\mathcal{N}\left(\left(DF_\mathrm{red}(u)\right)^\ast\right)=0$ and $U^\ast=\mathcal{R}\left(\left(DF_\mathrm{red}(u)\right)^\ast\right)$ for all $u\approx u_0$. In particular, there exists a unique $z\in \left(Y\times\mathbb{R}^{|\mathbb{I}|}\times\mathbb{R}^{|\mathbb{P}|}\right)^\ast$ such that $\left(D\Psi_1(u_0)\right)^\ast=\left(DF_\mathrm{red}(u_0)\right)^*z$. For every $\sigma_\mathbb{P}$ with $\|\sigma_\mathbb{P}\|=\epsilon\ll 1$, there exists a unique solution near $u_0$ to the first three equations in \eqref{lemeq}. At each such point, the fourth equation may be solved for $\lambda$, $\eta_\mathbb{I}$, and $\sigma_\mathbb{P}$ for given $\eta_1> 0$ such that the vector $(-\lambda/\eta_1,-\eta_\mathbb{I}/\eta_1,-\sigma_\mathbb{P}/\eta_1)$ is close to $z$. It follows that there exists a least one such point where the two values of the vector $\sigma_\mathbb{P}$ agree for some $0<\eta_1\ll 1$. The claim follows by considering variations in $\epsilon$.
\end{proof}

\begin{lemma}\label{lem4}
Suppose that $\{i:G_i(u_0)=0\}=k$, the linear map $DF_\mathrm{red}(u_0)$ is full rank with one-dimensional nullspace, and $\left(D\Psi_1(u_0)\right)^\ast$ and $\left(DG_k(u_0)\right)^\ast$ are linearly independent of $\left(DF_\mathrm{red}(u_0)\right)^*$. Then, $(u_0,0,0,0,\Psi_1(u_0),\Psi_{\mathbb{J}}(u_0),0,0)$ lies on a one-dimensional manifold of solutions to the continuation problem obtained by substituting $G_k(u)=0$ for the corresponding nonlinear complementary condition in $F_\mathrm{rest}=0$. This manifold is locally parameterized by $\eta_1$, and $\sigma_k,\sigma_\mathbb{P}\ne 0$ for $\eta_1$ close, but not equal to $0$.
\end{lemma}
\begin{proof}
By assumption, there exists a unique $(z,\zeta)\in \left(Y\times\mathbb{R}^{|\mathbb{I}|}\times\mathbb{R}^{|\mathbb{P}|}\times\mathbb{R}\right)^\ast$ such that $\left(D\Psi_1(u_0)\right)^\ast=\left(DF_\mathrm{red}(u_0)\right)^*z+\left(DG_k(u_0)\right)^\ast\zeta$. For $u\approx u_0$, roots of the modified continuation problem correspond to solutions to the system of equations obtained by adding $\left(DG_k(u)\right)^\ast\sigma_k$ to the left-hand side of the last equation in \eqref{lemeq} and appending $G_k(u)=0$. For every $\sigma_\mathbb{P}$ with $\|\sigma_\mathbb{P}\|=\epsilon\ll 1$, there exists a unique one-dimensional manifold of solutions to the first three equations in \eqref{lemeq}. By continuity, each such manifold generically contains a unique intersection with $G_k=0$. At each such point, the remaining equation may be solved for $\lambda$, $\eta_\mathbb{I}$, $\sigma_\mathbb{P}$, and $\sigma_k$ for given $\eta_1> 0$ such that the vector $(-\lambda/\eta_1,-\eta_\mathbb{I}/\eta_1,-\sigma_\mathbb{P}/\eta_1,-\sigma_k/\eta_1)$ is close to $(z,\zeta)$. It follows that there exists a least one such point where the two values of the vector $\sigma_\mathbb{P}$ agree for some $0<\eta_1\ll 1$. The claim follows by considering variations in $\epsilon$.
\end{proof}
The point $u_0$ in the lemma is a singular point of the restricted continuation problem $F_\mathrm{rest}=0$, but a regular solution point of the modified continuation problem constructed in the lemma. Along the solution manifold to the latter problem, $\sigma_k$ is typically positive only on one side of $u_0$. It follows that two solution manifolds of the restricted continuation problem terminate at $u_0$, one in $G_k=0$ (with $\sigma_k>0$) and one away from $G_k=0$ (with $\sigma_k=0$). Numerical parameter continuation may switch between these manifolds, effectively bypassing the singular point at $u_0$, provided that the tangent directions are positively aligned.

%\section*{Acknowledgments}

\bibliographystyle{siamplain}
\bibliography{references}

\end{document}

%% file: ex_shared.tex
\usepackage{lipsum}
\usepackage{amsfonts}
\usepackage{graphicx}
\usepackage{epstopdf}
\usepackage{stmaryrd}
\usepackage{mathabx}
\usepackage{algorithmic}
\usepackage{tikz}
\usepackage{tikz-qtree}
\ifpdf
  \DeclareGraphicsExtensions{.eps,.pdf,.png,.jpg}
\else
  \DeclareGraphicsExtensions{.eps}
\fi

\usepackage{enumitem}
\setlist[enumerate]{leftmargin=.5in}
\setlist[itemize]{leftmargin=.5in}

\newcommand{\TheTitle}{ Optimization with Equality and Inequality Constraints Using Parameter Continuation} 
\newcommand{\TheAuthors}{M. Li, H. Dankowicz}

\newsiamremark{remark}{Remark}
\newsiamremark{hypothesis}{Hypothesis}
\crefname{hypothesis}{Hypothesis}{Hypotheses}
\newsiamthm{claim}{Claim}

\headers{Constrained Optimization using Parameter Continuation}{\TheAuthors}

\title{{\TheTitle}\thanks{This paper was submitted on August 16, 2018}} %\funding{}
\author{
  Mingwu Li\thanks{Department of Mechanical Science and Engineering, University of Illinois at Urbana-Champaign, Urbana, IL  
    (\email{mingwul2@illinois.edu}, \email{danko@illinois.edu}).}
  \and
  Harry Dankowicz\footnotemark[2]
}

\usepackage{amsopn}

\DeclareMathOperator{\sgn}{sgn}